\documentclass[11pt, twoside]{article}

\usepackage[english]{babel}

\usepackage{amssymb}
\usepackage{mathrsfs}
\usepackage{amsmath}
\usepackage{amsthm}
\usepackage{amsfonts}
\usepackage{latexsym}
\usepackage{indentfirst}
\usepackage{color}
\usepackage{txfonts}
\usepackage{enumerate}

\usepackage[colorlinks=true,
linkcolor=blue,
citecolor=red,
urlcolor=magenta,
backref=page
]{hyperref}

\usepackage{txfonts}

\usepackage{anysize}

\allowdisplaybreaks

\newtheorem{theorem}{Theorem}[section]
\newtheorem{lemma}[theorem]{Lemma}

\newtheorem{proposition}[theorem]{Proposition}

\theoremstyle{definition}
\newtheorem{remark}[theorem]{Remark}

\newtheorem{definition}[theorem]{Definition}
\newcounter{assum}

\renewcommand{\appendix}{\par
\setcounter{section}{0}%
\setcounter{subsection}{0}%
\setcounter{subsubsection}{0}%
\gdef\thesection{\@Alph\c@section}%
\gdef\thesubsection{\@Alph\c@section.\@arabic\c@subsection}%
\gdef\theHsection{\@Alph\c@section.}%
\gdef\theHsubsection{\@Alph\c@section.\@arabic\c@subsection}%
\csname appendixmore\endcsname
}

\numberwithin{equation}{section}

\begin{document}
\title{\bf\Large
Compactness and Its Applications of Sobolev Spaces Associated 
with Ball Banach Function Spaces\footnotetext{\hspace{-0.35cm} 2020 {\it
Mathematics Subject Classification}.
Primary 46E35; Secondary 26D10, 42B25, 42B35, 35R11.
\endgraf {\it Key words and phrases.}
ball Banach function space, Sobolev space, compactness,
Poincar\'e inequality, well-posedness.}}
\author{Xiaosheng Lin, Dachun Yang\footnote{Corresponding
author, E-mail: \texttt{dcyang@bnu.edu.cn}/{\color{red}\today}/Final version.},\  \ Wen Yuan
and Yangyang Zhang}
\date{}
\maketitle

\vspace{-0.7cm}

\begin{center}
\begin{minipage}{13cm}
{
\small {\bf Abstract.}\quad
Let $N\in\mathbb{N}\cap[2,\infty),$
$\Omega$ be a bounded Lipschitz domain in $\mathbb{R}^N$,
and $X(\Omega)$ be a ball Banach function space on $\Omega.$
In this article, under some mild assumptions, we establish
a compactness theorem for Sobolev spaces associated with
$X(\Omega)$. Different from the fractional Sobolev space,
our proof is based on
an elaborate
 decomposition of bounded Lipschitz
domains and its corresponding weighted fractional Poincar\'e
inequality on each piece.
As applications, we obtain the fractional
Poincar\'e inequality in $X(\Omega)$
that, for any $s\in (s_0,1)$ and $f\in X(\Omega)$,
\begin{align*}
\|f-f_\Omega\|_{X(\Omega)}
\lesssim(1-s)^{\frac 1q}
\left\|\left[\int_\Omega
\frac{|f(\cdot)-f(y)|^q}{|\cdot-y|^{N+sq}}\,dy
\right]^{\frac{1}{q}}
\right\|_{X(\Omega)},
\end{align*}
where $s_0$ is a given positive constant and the implicit positive
constant is independent of $s$ and $f$.
Using this, we further establish the well-posedness of a weighted
Triebel--Lizorkin type nonlocal variational problem.
These results are of wide generality and, even when they
are applied to Morrey spaces,
weighted Lebesgue spaces,
mixed-norm Lebesgue spaces, variable Lebesgue spaces, Orlicz spaces, and
Orlicz-slice spaces, the obtained results are also new.}
\end{minipage}
\end{center}

\vspace{0.1cm}

\tableofcontents

\section{Introduction}
Let $N\in\mathbb{N}\cap [2,\infty)$ and
$\Omega\subset\mathbb{R}^{N}$ be a bounded Lipschitz domain.
Recall that, for any given $p\in[1,\infty)$,
the \emph{Sobolev space} $W^{1,p}(\Omega)$
is defined by setting
\begin{align*}
W^{1,p}(\Omega):=\left\{f\in L^p(\Omega):\ |\nabla f|\in L^p(\Omega)\right\},
\end{align*}
where $\nabla f$ denotes the \emph{weak gradient} of $f$.
Let $\{\rho_n\}_{n\in\mathbb{N}}$ be a sequence of radial functions on $\mathbb{R}^{N}$
such  that, for any $n\in\mathbb{N}$ and $\delta\in (0,\infty)$,
$\rho_n\geq 0$,
\begin{align}\label{wuhu}
\int_{\mathbb{R}^{N}}\rho_{n}(|h|)\,dh=1,
\ \ \text{and}\ \ \lim_{n\to\infty}\int_{|h|>\delta}\rho_n(|h|)\,dh=0.
\end{align}
Let $p\in[1,\infty)$. In the seminal articles \cite{bbm01, b02}, Bourgain, Brezis, and Mironescu
showed that,
for any $f\in W^{1,p}$,
\begin{align*}
\lim_{n\to\infty}\int_{\Omega}\int_{\Omega}\frac{|f(x)-f(y)|^p}{|x-y|^p}\rho_n(x-y)\,dx\,dy
=K(p,N)\|\,|\nabla f|\,\|_{L^p(\Omega)}^p,
\end{align*}
where
\begin{align}\label{kchangshu}
K(p,N)
:=
\int_{S^{N-1}}|\omega\cdot e|^p\,d\sigma(\omega).
\end{align}
Since then, this result, now called the
Bourgain--Brezis--Mironescu (BBM) formula,
has attracted a lot of attention; see, for example,
\cite{bn16,dlyyz22,dm22,dm23,dssvy}.
In spirit of this formula,
Ponce \cite{p04} established a powerful
compactness result on Sobolev spaces. To be precise,
let $\{f_n\}_{n\in\mathbb{N}}\subset L^p(\Omega)$ be a bounded sequence such that
\begin{align}\label{haqi}
\sup_{n\in\mathbb{N}}\int_{\Omega}\int_{\Omega}
\frac{|f_n(x)-f_n(y)|^p}{|x-y|^p}\rho_n(x-y)\,dx\,dy<\infty.
\end{align}
Ponce \cite[Theorem 1.2]{p04} showed  that $\{f_n\}_{n\in\mathbb{N}}$
is relatively compact in $L^p(\Omega).$
Using this Ponce \cite[Theorem 1.1]{p04} obtained an improved
Poincar\'e inequality.

The key point in the proof of \cite[Theorem 1.2]{p04}
is the handling of the boundary. After
obtaining local compactness in the whole space $\mathbb{R}^{N}$, Ponce
established an $L^p$-estimate near the boundary of $\Omega.$
Roughly speaking,
this estimate shows that the
$L^p$ norm over the entire domain $\Omega$ can be controlled by both the $L^p$ norm on a
slightly smaller interior subdomain and the BBM norm
\eqref{haqi}. This estimate
then  allows one to extend
interior compactness to  compactness
on the whole bounded Lipschitz domain.
However, its proof  relies heavily on specific properties
of $L^p$ spaces, in particular polar decomposition and translation
invariance, which are not available in general function spaces.

Recall that the ball quasi-Banach function space
was introduced by Sawano et al. \cite{shyy} in order to
handle, in a unified way, many important function spaces arising in harmonic
analysis and partial differential equations, such as Morrey spaces,
Orlicz spaces, and variable Lebesgue spaces. It is worth pointing
out that this framework is particularly flexible, since it also contains several
spaces which are not covered by the classical Banach function spaces of
Bennett and Sharpley \cite{bs88}.

Let $q\in[1,\infty)$,
$R\in (\mathrm{diam}\,(\Omega),\infty)$, and $\{s_n\}_{n\in\mathbb{N}}\subset (0,1)$
be such that $\lim_{n\to\infty}s_n=1.$ For any $n\in\mathbb{N}$ and $h\in\mathbb{R}^{N},$ let
\begin{align}\label{tesuhe}
\rho_n(h):= \frac{q(1-s_n)}{\sigma(S^{N-1}) R^{q(1-s_n)}} |h|^{q(1-s_n)-N} \mathbf{1}_{\{|h|<R\}}.
\end{align}
It is easy to verify that $\{\rho_n\}_{n\in\mathbb{N}}$ satisfy \eqref{wuhu}.
The main purpose of this article is to establish the following compactness
theorem of ball Banach function spaces for the  kernels $\{\rho_n\}_{n\in\mathbb{N}}$
in \eqref{tesuhe}, whose proof is given in Subsection \ref{sec2.4}.
(In the remainder of this introduction, we need to use some related notation and concepts
from Subsection \ref{sec2.1}.)

\begin{theorem}\label{thm-domain}
Let $N\in\mathbb{N}\cap [2,\infty)$ and $\Omega\subset\mathbb R^N$
be a bounded Lipschitz domain, and let $q,r\in[1,\infty)$.
Let $X$ and $X^{\frac{1}{r}}$ be a Ball Banach function space.
Assume that $X'$ has an
absolutely continuous norm and the Hardy--Littlewood
maximal operator $\mathcal{M}$ is bounded on $(X^{\frac{1}{r}})'$.
Let $\{f_n\}_{n\in\mathbb N}\subset X(\Omega)$
and $\{s_n\}_{n\in\mathbb{N}}\subset (0,1)$
satisfy $\lim_{n\to\infty}s_n=1$,
$\sup_{n\in\mathbb N}\|f_n\|_{X(\Omega)}<\infty,$
and
\begin{equation}\label{youjie2}
\sup_{n\in\mathbb{N}}(1-s_{n})^{\frac{1}{q}}\left\|
\left[
\int_\Omega
\frac{|f_{n}(\cdot)-f_{n}(y)|^q}{|\cdot-y|^{N+s_{n}q}}\,dy
\right]^{\frac1q}
\right\|_{X(\Omega)}<\infty.
\end{equation}
Then $\{f_n\}_{n\in\mathbb{N}}$ is relatively compact in $X(\Omega)$. Moreover, if $f_{n_k}\to f$ in $X(\Omega)$ as $k\to\infty$, then $f\in W^{1,X}(\Omega)$ and, for any $j\in\{1,\ldots,N\},$
\begin{align*}
\sum_{j=1}^N\left\|\partial^j f\right\|_{X(\Omega)}
\leq
B_{q,N}\varliminf_{k\to\infty}
(1-s_{n_k})^{\frac{1}{q}}\left\|
\left[
\int_\Omega
\frac{|f_{n_k}(\cdot)-f_{n_k}(y)|^q}{|\cdot-y|^{N+s_{n_k}q}}\,dy
\right]^{\frac1q}
\right\|_{X(\Omega)},
\end{align*}
where
\begin{align}\label{e1.6}
B_{q,N}:=\frac{2N[\sigma(S^{N-1})]^{1-\frac{1}{q}}q^{\frac{1}{q}}}{K(1,N)}.
\end{align}
\end{theorem}
Such an extension is far from straightforward. Indeed, in general ball
Banach function spaces, neither rotation invariance nor translation
invariance is available, while both properties are essentially used in
the proof of \cite[Theorem 1.2]{p04}. To avoid these, we first carefully decompose
the domain $\Omega$ into finitely many
interior cubes and boundary pieces which, after a rigid motion, become
Lipschitz images of half-cubes (see Lemma \ref{lem:decom}). On these local pieces,
we establish a weighted fractional Poincar\'e inequality
by means of sparse domination. Applying these, via introducing
a finite-rank averaging operator
which consists of the averages over the pieces of
the aforementioned
decomposition, yields the desired compactness.

Theorem \ref{thm-domain} has the following two applications.

\textbf{Application (I): Improved fractional Poincar\'e
inequality on bounded Lipschitz domains.}
In recent years,
(fractional) Poincar\'e inequalities have attracted considerable attention,
especially in weighted Lebesgue spaces and
doubling metric measure spaces;
see, for example, \cite{fpw98,hmpv23,hk25,kklv19,klv21,klvz19,
lw26arxiv,mp98,mpw24}.
Applying Theorem \ref{thm-domain}, we establish a fractional
Poincar\'e inequality in ball Banach function spaces on bounded Lipschitz domains,
which is stronger than the first-order Poincar\'e inequality.
Applying this to weighted Lebesgue spaces, we further derive a weighted
fractional Poincar\'e inequality on bounded Lipschitz domains.
\begin{theorem}\label{thm:frac-poincare}
Let $N\in\mathbb{N}\cap [2,\infty)$, $\Omega\subset\mathbb R^N$ be a bounded
Lipschitz domain, and  $X$ be a Ball Banach function space on $\mathbb{R}^{N}$.
Let $q,r\in[1,\infty)$. Assume that $X'$ has an
absolutely continuous norm and the Hardy--Littlewood maximal
operator $\mathcal{M}$ is bounded on $(X^{\frac{1}{r}})'$.
Let $C_{X,\Omega}$ be as in Proposition \ref{x-poincare} and
$D\in (C_{X,\Omega}B_{q,N},\infty)$ with $B_{q,N}$ as in \eqref{e1.6}.
Then there exists $s_0\in(0,1)$, depending on $D$,
such that, for any $s\in (s_0,1)$ and $f\in X(\Omega)$,
\begin{equation}\label{eq:fractional}
\|f-f_\Omega\|_{X(\Omega)}
\leq D(1-s)^{\frac 1q}\left\|
\left[\int_\Omega
\frac{|f(\cdot)-f(y)|^q}{|\cdot-y|^{N+sq}}\,dy
\right]^{\frac1q}
\right\|_{X(\Omega)}.
\end{equation}
\end{theorem}
\begin{remark} Let the notation be the same as in Theorem \ref{thm:frac-poincare}.
Assume further that $\mathcal{M}$ is bounded on $(X^{\frac{1}{q}})'$.
By Proposition \ref{x-poincare}, we obtain the first-order Poincar\'e inequality that, for any
$f\in W^{1,X}(\Omega),$
\begin{align}\label{yingying}
\|f-f_\Omega\|_{X(\Omega)}
\leq
C_{X,\Omega}\|\,|\nabla f|\,\|_{X(\Omega)}.
\end{align}
It is worth pointing out that \eqref{eq:fractional} is stronger than \eqref{yingying}.
Indeed, using \cite[Theorem 5.13]{zyyjga}
and repeating the proof of \cite[Proposition 3.9]{zyyjga}, we conclude that, for any
 $s\in (0,1)$
and $f\in W^{1,X}(\Omega),$
\begin{align*}
(1-s)^{\frac 1q}\left\|\left[\int_\Omega
\frac{|f(\cdot)-f(y)|^q}{|\cdot-y|^{N+sq}}\,dy
\right]^{\frac1q}
\right\|_{X(\Omega)}
\lesssim
\|\,|\nabla f|\,\|_{X(\Omega)}
\end{align*}
with the implicit positive constant independent of $s$ and $f$.
\end{remark}
Applying Theorem \ref{thm:frac-poincare} to weighted Lebesgue spaces,
we further obtain the following conclusion, which when $\Omega$ is a cube
was obtained by Lorist and Wagenaar in \cite[Corollary 5.4]{lw26arxiv}.
\begin{theorem}\label{thm:weighted-poincare}
Let $N\in\mathbb{N}\cap [2,\infty)$ and
$\Omega\subset\mathbb{R}^{N}$ be a bounded Lipschitz domain.
Let $p\in (1,\infty)$ and $\omega\in A_{p}(\mathbb{R}^N)$. Let $q\in [1,\infty)$.
Then there exist a positive constant $E$ and $s_0\in (0,1)$ such that,
for any $s\in (s_0,1)$ and $f\in L^p_{\omega}(\Omega)$,
\begin{equation*}
\|f-f_\Omega\|_{L^p_{\omega}(\Omega)}
\leq E(1-s)^{\frac{1}{q}}\left\{
\int_{\Omega}\left[\int_\Omega\frac{|f(x)-f(y)|^q}{|x-y|^{N+sq}}
\,dy\right]^{\frac{p}{q}}\omega(x)\,dx\right\}^{\frac{1}{p}}.
\end{equation*}
\end{theorem}
The proofs of Theorems \ref{thm:frac-poincare} and \ref{thm:weighted-poincare} are
given in Subsection \ref{pro:1.2}.

\textbf{Application (II): Well-posedness of a weighted Triebel--Lizorkin
type nonlocal variational problem.}
Let $N\in\mathbb{N}\cap [2,\infty)$ and $\Omega\subset\mathbb{R}^{N}$
be a bounded Lipschitz domain. Let $s\in (0,1)$ and $p\in (1,\infty)$.
Recall that the \emph{regional fractional $p$-Laplacian $(-\Delta)^s_{p,\Omega}$}
is formally defined by setting, for any suitable function $u$ and any $x\in\Omega$,
\begin{align*}
(-\Delta)^s_{p,\Omega}u(x)
:=2\,\mathrm{p.v.}\int_\Omega
\frac{|u(x)-u(y)|^{p-2}[u(x)-u(y)]}{|x-y|^{N+sp}}\,dy.
\end{align*}
This operator $(-\Delta)^s_{p,\Omega}$ is also the \emph{Euler--Lagrange
operator} associated with the energy
\begin{align*}
\frac{1}{p}
\int_\Omega\int_\Omega
\frac{|u(x)-u(y)|^p}{|x-y|^{N+sp}}\,dy\,dx
\end{align*}
on $W^{s,p}(\Omega).$
The strong version of the regional fractional
$p$-Laplace equation is
\begin{align*}
(-\Delta)^s_{p,\Omega}u=f\ \ \text{in}\ \ \Omega.
\end{align*}
For more studies about the $p$-Laplacian, we refer to
\cite{bgk22,bgk23,cmn13,cr22,crr18,gk22,gs25,ms18, w15,w16}.

Now, we consider the Euler--Lagrange equation in weighted
Triebel--Lizorkin spaces,
which is a generalization of the $p$-Laplace equation. Let $s\in (0,1)$, $p,q\in (1,\infty)$, $\omega\in A_p(\mathbb{R}^N)$, and
$F^{s,\omega}_{p,q,0}(\Omega)$ be the weighted Triebel--Lizorkin
space with zero average (see Definition \ref{kongjian} for its precise definition).
For any $u\in F^{s,\omega}_{p,q,0}(\Omega),$ define the energy
\begin{align}\label{1.7x}
\Phi_s(u):=
\frac{1}{p}\int_\Omega
\left[\int_\Omega\frac{|u(x)-u(y)|^q}{|x-y|^{N+sq}}\,d y\right]^{\frac{p}{q}}
\omega(x)\,dx.
\end{align}
For any bounded measurable set $U\subset \mathbb{R}^{N}$, any  $f\in L^1_{\mathrm{loc}}$, and any $x\in\mathbb{R}^N$, let
\begin{align}\label{1.7y}
\mathcal{D}^{s}_{q,U}f(x):=
\left[\int_U
\frac{|f(x)-f(y)|^q}{|x-y|^{N+sq}}\,dy
\right]^{\frac{1}{q}}.
\end{align}
Then the formal \emph{Euler--Lagrange operator $\mathcal{L}^{s,\omega}_{p,q}$}
associated with $\Phi_s$ is given by
\begin{align*}
\mathcal{L}^{s,\omega}_{p,q}u(x)&:=\mathrm{p.v.}\int_{\Omega}
\frac{|u(x)-u(y)|^{q-2}[u(x)-u(y)]}{|x-y|^{N+sq}}\\
&\quad\times
\left\{\omega(x)\left[\mathcal{D}^s_{q,\Omega}u(x)\right]^{p-q}
+\omega(y)\left[\mathcal{D}^s_{q,\Omega}u(y)\right]^{p-q}\right\}\,dy.
\end{align*}
When $p=q$ and $\omega\equiv1$,
$\mathcal{L}^{s,\omega}_{p,q}=(-\Delta)^s_{p,\Omega}.$
Next, we prove the existence and uniqueness of weak
solutions to the following weighted nonlocal problem
\begin{align}\label{qiangjie}
\mathcal{L}^{s,\omega}_{p,q}u=f\ \ \text{in}\ \ \Omega.
\end{align}
To this end, for any $u,v\in F^{s,\omega}_{p,q,0}(\Omega)$ with $u\not\equiv0$,
 let $A_s(0,v):=0$  and
\begin{align*}
A_s(u,v)
:&=
\int_\Omega
\omega(x)
\left[
\int_\Omega
\frac{|u(x)-u(y)|^q}{|x-y|^{N+sq}}\,dy
\right]^{\frac{p}{q}-1}
\\
&\quad\times
\int_\Omega
\frac{|u(x)-u(y)|^{q-2}[u(x)-u(y)][v(x)-v(y)]}{|x-y|^{N+sq}}
\,dy\,dx.
\end{align*}

Applying Theorem \ref{thm:weighted-poincare}, we obtain the following theorem.
\begin{theorem}\label{thm:numan}
Let $N\in\mathbb{N}\cap [2,\infty)$ and $\Omega\subset\mathbb{R}^{N}$
be a bounded Lipschitz domain. Let $p,q\in (1,\infty)$. Assume that $\omega\in A_{p}(\mathbb{R}^N)$.
Let $E,$ $s_0,$ and $s$ be as in Theorem \ref{thm:weighted-poincare}, and let $f\in L^{p'}_{\omega^{1-p'}}(\Omega)$ satisfy $\int_{\Omega}f(x)\,dx=0.$
Then \eqref{qiangjie} has a
unique weak solution in $F^{s,\omega}_{p,q,0}(\Omega)$, i.e.,
there exists a unique
$u\in F^{s,\omega}_{p,q,0}(\Omega)$ such that, for any $v\in F^{s,\omega}_{p,q,0}(\Omega),$
\begin{equation}\label{eq:weak}
A_s(u,v)=\int_\Omega f(x)v(x)\,dx.
\end{equation}
Moreover,
\begin{align}\label{huxiao}
\|u\|_{F^{s,\omega}_{p,q}(\Omega)}
\leq
\left[E^{\frac{1}{p-1}}
(1-s)^{\frac{1}{q(p-1)}}+E^{\frac{p}{p-1}}(1-s)^{\frac{p}{q(p-1)}}\right]
\|f\|_{L^{p'}_{\omega^{1-p'}}(\Omega)}^{\frac{1}{p-1}}.
\end{align}
\end{theorem}
The proof of Theorem \ref{thm:numan} is given in Subsection \ref{sec2.6}.

The remainder of this article is organized as follows.

In Section \ref{sec2}, we  prove our main results. To be precise, in
Subsection \ref{sec2.1}, we first recall some basic notions and  known facts on ball
quasi-Banach function spaces.
In Subsection \ref{sec2.2}, we establish an elaborate decomposition
of bounded Lipschitz domains.
This decomposition is one of
the main technical ingredients of this article and plays a key role in
the proof of Theorem \ref{thm-domain}.
In Subsection \ref{sec2.3}, we give a
weighted fractional Poincar\'e inequality on cubes.
In Subsection \ref{sec2.4}, using the preceding
decomposition and the weighted fractional Poincar\'e inequality, we prove
Theorem \ref{thm-domain}.
In Subsection \ref{pro:1.2}, as a consequence
of Theorem \ref{thm-domain}, we obtain two improved fractional Poincar\'e
inequalities in Theorems \ref{thm:frac-poincare} and  \ref{thm:weighted-poincare}. Finally, in
Subsection \ref{sec2.6},
we use
the variational method and
Lemma \ref{lem:bianfen}
to prove
Theorem \ref{thm:numan}.

In Section \ref{shijie}, we apply Theorems \ref{thm-domain}
and \ref{thm:frac-poincare}, respectively, to
the Morrey space,
the variable Lebesgue space,
the mixed-norm Lebesgue space,
the weighted Lebesgue space,
the Orlicz space, and
the Orlicz-slice space. All these obtained
results are new. Owing to their generality and flexibility
of Theorems \ref{thm-domain} and \ref{thm:frac-poincare},
further applications can be expected.

At the end of this section, we make some notational conventions. Let
${\mathbb N}:=\{1,2,\ldots\}$ and ${\mathbb Z}_+:={\mathbb N}\cup\{0\}$.
For any $\beta:=(\beta_1,\ldots,\beta_n)\in\mathbb{Z}_+^n:=(\mathbb{Z}_+)^n$,
let $|\beta|:=\beta_1+\cdots+\beta_n$
and, for any $x:=(x_1,\ldots,x_n),$ let
$x^\beta:=x_1^{\beta_1}\cdots x_n^{\beta_n}$ and
 $\partial^\beta:=(\frac{\partial}{\partial x_1})^{\beta_1}\cdots(\frac{\partial}{\partial x_n})^{\beta_n}$.
We always denote by $C$ a positive constant
which is independent of the main parameters involved,
but it may vary from line to line.
We also use $C_{\alpha,\beta,...}$ to denote a
positive constant depending on the indicated parameters $\alpha$, $\beta\ldots$.
The notation $f\lesssim g$ means that $f\le Cg$.
If $f\lesssim g$ and $g\lesssim f$,
we then write $f\sim g$. If $f\le Cg$ and $g=h$
or $g\le h$, we then write $f\lesssim g=h$ or $f\lesssim g\le h$.
For any $x\in{\mathbb{R}^n}$ and $r\in(0,{\infty})$,
we denote by
$B(x,r):=\left\{y\in{\mathbb{R}^n}:\ |y-x|<r\right\}$
the ball with center $x$ and radius $r$.
Also, for any $f\in L^1_{\mathrm{loc}}$,
its \emph{Hardy--Littlewood maximal function} $\mathcal{M}f$
is defined by setting, for any $x\in\mathbb{R}^n$,
\begin{align*}
\mathcal{M}f(x):=\sup_{B\ni x}
\frac{1}{|B|}\int_B\left|f(y)\right|\,dy,
\end{align*}
where the supremum is taken over all
balls $B\subset\mathbb{R}^n$ containing $x$.
We use $\mathbf{1}_E$ to denote the
characteristic function of a measurable set $E\subset {\mathbb{R}^n}$
and $\mathbf{0}$ to denote the \emph{origin} of ${\mathbb{R}^n}$.
For any $p\in[1,\infty]$,
we denote by $p'$ its \emph{conjugate index};
that is, $\frac{1}{p}+\frac{1}{p'}=1$.
We also use $\varliminf$ to denote the \emph{limit
inferior}. Finally, in all proofs we
consistently retain the notation
introduced in the original theorem (or related statement).

\section{Proofs of Main Results}\label{sec2}

This section consists of six subsections.
In Subsection \ref{sec2.1} we recall several basic concepts
and known facts on ball quasi-Banach function spaces.
In Subsection \ref{sec2.2}, we give an elaborate decomposition
of bounded Lipschitz domains, which is the main technical
tool used in the proof of Theorem \ref{thm-domain}.
Subsection \ref{sec2.3} is devoted to the
weighted fractional Poincar\'e inequality on cubes.
Using this and the decomposition
obtained in Subsection \ref{sec2.2}, we give the proof of
Theorem \ref{thm-domain}
in Subsection \ref{sec2.4}.
As applications of Theorem \ref{thm-domain}, in Subsection \ref{pro:1.2}
we establish the fractional Poincar\'e inequality stated in Theorems \ref{thm:frac-poincare} and \ref{thm:weighted-poincare}.
Finally, in Subsection \ref{sec2.6} we  prove Theorem \ref{thm:numan} by means of the variational method.

\subsection{Preliminary}\label{sec2.1}
We begin by  recalling some preliminaries on ball
quasi-Banach function spaces introduced in \cite{shyy}.
Let
$\mathbb{B}
:=\left\{B(x,r):\ x\in\mathbb{R}^n \text{ and } r\in(0,\infty)\right\}.$
Denote by the \emph{symbol} $\mathscr M$ the
set of all measurable functions on $\mathbb{R}^n$.
\begin{definition}\label{Debqfs}
A quasi-Banach space $X\subset{\mathscr M}$, equipped with
a quasi-norm
$\|\cdot\|_X$ which makes sense for
all functions in ${\mathscr M}$,
is called a \emph{ball quasi-Banach
function space} (for short, BQBF \emph{space}) if it satisfies that
\begin{itemize}
\item[(i)] for any $f\in {\mathscr M}$,
$\|f\|_X=0$ implies that $f=0$ almost everywhere,
\item[(ii)] for any $f,g\in {\mathscr M}$,
$|g|\le |f|$ almost everywhere implies
that $\|g\|_X\le\|f\|_X$,
\item[(iii)] for any $\{f_m\}_{m\in\mathbb{N}}
\subset {\mathscr M}$
and $f\in {\mathscr M}$, $0\le f_m\uparrow f$
almost everywhere as $m\to\infty$ implies that
$\|f_m\|_X\uparrow\|f\|_X$ as $m\to\infty$,
\item[(iv)] $B\in\mathbb{B}$ implies that $\mathbf
{1}_B\in X$.
\end{itemize}
Moreover, a ball quasi-Banach function space
$X$ is called a \emph{ball Banach function space} (for short, BBF \emph{space})
if the norm of $X$ satisfies the triangle
inequality: for any $f,g\in X$,
\begin{align*}
\|f+g\|_X\le \|f\|_X+\|g\|_X,
\end{align*}
and that, for any $B\in \mathbb{B}$, there exists a
positive constant $C_{(B)}$, depending on $B$, such that,
for any $f\in X$,
\begin{equation*}
\int_B|f(x)|\,dx\le C_{(B)}\|f\|_X.
\end{equation*}
\end{definition}
\begin{remark}
\begin{enumerate}
\item[$\mathrm{(i)}$] Let $X$ be a BQBF space on $\mathbb{R}^n$. By \cite[Remark 2.6(i)]{yhyy23},
we conclude that, for any
$f\in\mathscr{M}$, $\|f\|_{X}=0$ if and only if $f=0$
almost everywhere.

\item[$\mathrm{(ii)}$] Using \cite[Theorem 2]{dfmn21},
we conclude that both (ii) and (iii) of
Definition \ref{Debqfs} imply that any ball quasi-Banach
function space is complete.

\item[$\mathrm{(iii)}$] As mentioned in
\cite[Remark 2.6(ii)]{yhyy23}, we obtain an
equivalent formulation of Definition \ref{Debqfs}
via replacing ``any ball $B$" by ``any
bounded measurable set $E$" therein.

\item[$\mathrm{(iv)}$] In Definition \ref{Debqfs}, if we
replace ``any ball $B$" by ``any measurable set $E$ with
finite measure", then we obtain the
definition of (quasi-)Banach function spaces
introduced in \cite[Definitions 1.1 and 1.3]{bs88}. Thus,
a (quasi-)Banach function space
is always a ball (quasi-)Banach function
space and the converse may not be true.

\item[\rm(v)]
In Definition \ref{Debqfs},
if we replace (iv)
by the \emph{saturation property} that,
for any measurable set $E\subset\mathbb{R}^n$
with $|E|\in(0,\infty)$, there exists a measurable subset $F$ of $E$
with $|F|\in(0,\infty)$ such
that $\mathbf{1}_F\in X$,
then we obtain the concept of quasi-Banach function spaces
in Lorist and Nieraeth
\cite{ln}.
Moreover, by \cite[Proposition 4.21]{n23},
we find that, if the
quasi-normed vector space $X$
under consideration satisfies
the additional assumption that
the Hardy--Littlewood maximal operator
$\mathcal{M}$ is weakly bounded on $X$ or
one of its convexification,
then the definition of quasi-Banach function
spaces in \cite{ln} coincides with
the definition of ball quasi-Banach function spaces. Thus,
under this additional assumption,
working either with ball quasi-Banach function
spaces in the sense of Definition \ref{Debqfs}
or with quasi-Banach function spaces in
the sense of \cite{ln} would yield
exactly the same results.
\end{enumerate}
\end{remark}
We recall the concept of the associate space and the convexification of a ball Banach function space,
which can be found, for instance,
in \cite[Chapter 1, Definitions 2.1 and 2.3]{bs88} and \cite[Definition 2.6]{shyy}.
\begin{definition}
For any BBF space $X$,
the \emph{associate space} (also called the
\emph{K\"othe dual}) $X'$ is defined by setting
\begin{equation}\label{asso}
X':=\left\{f\in\mathscr M:\ \|f\|_{X'}
:=\sup_{\{g\in X:\ \|g\|_X=1\}}\|fg\|_{L^1}
<\infty\right\},
\end{equation}
where $\|\cdot\|_{X'}$ is called the
\emph{associate norm} of $\|\cdot\|_X$.
\end{definition}
\begin{definition}
Let $p\in(0,\infty)$ and $X$ be a $\mathrm{BQBF}$ space.
The \emph{p-convexification} $X^p$ of $X$ is defined by setting
	\begin{equation*}
		X^p:=\left\{f\in\mathscr{M}:\ \|f\|_{X^p}:=\left\||f|^p\right\|_{X}^{\frac{1}{p}}<\infty\right\}.
	\end{equation*}
\end{definition}
The following H\"older's inequality is a simple
corollary of both Definition \ref{Debqfs}
and \eqref{asso} (see \cite[Theorem 2.4]{bs88}).
\begin{lemma}\label{xholder}
Let $X$ be a BBF space and $X'$ its associate space.
If $f\in X$ and $g\in X'$, then $fg$ is integrable and
\begin{equation*}
\int_{\mathbb{R}^n}|f(x)g(x)|\,dx\le \|f\|_X\|g\|_{X'}.
\end{equation*}
\end{lemma}
The following lemma is exactly \cite[Lemma 2.6]{zwyy}.
\begin{lemma}\label{xduiou}
Let $X$ be a BBF space. Then $X$ coincides with its
second associate space $X''$. In other words, a
function $f$ belongs to $X$ if and only if
it belongs to $X''$ and,
in that case,
$
\|f\|_X=\|f\|_{X''}.
$
\end{lemma}
We recall the   definition of ball Banach function
spaces with absolutely continuous norm;
see \cite[Definition 3.1]{bs88}.
\begin{definition}
A BBF space $X$ is said to have an
\emph{absolutely continuous norm} if, for any $f\in X$
and any sequence of measurable sets $\{E_j\}_{j\in\mathbb{N}}\subset \mathbb{R}^n$
satisfying that $\mathbf{1}_{E_j}\to 0$ almost everywhere as $j\to\infty$,
$\|f\mathbf{1}_{E_j}\|_X\to 0$ as $j\to\infty$.
\end{definition}
The following lemma can be   found in \cite[p.7, Corollary 4.3]{bs88}.
\begin{lemma}\label{xingxing}
Let $X$ be a BBF space.
Then $X$ has an absolutely continuous norm if and only if
$X^\ast=X',$ where $X^\ast$ denotes the space of all   bounded linear functionals on $X.$
\end{lemma}
\begin{definition}
Let $\Omega\subset \mathbb{R}^N$
be open.
The \emph{restricted space} $X(\Omega)$ of
the BQBF space $X$ to $\Omega$
is defined by setting
$X(\Omega):=\{f\in\mathscr{M}(\Omega):\
f=g|_\Omega\text{ for some }g\in X\}$.
For any $f\in X(\Omega)$, let
\begin{align}\label{1035}
\left\|f\right\|_{X(\Omega)}:=\inf\left\{\|g\|_{X}:\
f=g|_\Omega,\ g\in X\right\}.
\end{align}
\end{definition}

The following proposition shows that
the infimum in \eqref{1035} can be attained, which is precisely \cite[Proposition 2.7]{zyyjga}.
\begin{proposition}\label{norm}
Let $\Omega\subset \mathbb{R}^N$
be open. Let $X$ be a BQBF space
and $X(\Omega)$ its restricted space.
Then, for any $f\in X(\Omega)$,
$\|f\|_{X(\Omega)}=
\|\widetilde{f}\|_{X}$,
where
\begin{align*}
\widetilde{f}(x):=
\begin{cases}
f(x)&\textup{ if }x\in\Omega,\\
0&\textup{ if }x\in\Omega^\complement.
\end{cases}
\end{align*}
\end{proposition}
\begin{definition}
Let $X$ be a BQBF space. The \emph{inhomogeneous ball Banach
Sobolev space} $W^{1,X}(\Omega)$
is defined to be the set of all $f\in X(\Omega)$
satisfying that, for any $j\in\{1,\ldots,N\}$,
the $j^\mathrm{th}$-weak partial derivative $\partial_j f$ of $f$
exists and $\partial_j f\in X(\Omega)$.
Moreover, for any $f\in W^{1,X}(\Omega)$, let
\begin{align*}
\|f\|_{W^{1,X}(\Omega)}:=\|f\|_{X(\Omega)}+
\sum_{j=1}^N\left\|\partial_j f\right\|_{X(\Omega)}.
\end{align*}
\end{definition}

\subsection{Decomposition of Bounded Lipschitz Domains}\label{sec2.2}
In this subsection, we  give the following elaborate
decomposition of  bounded Lipschitz domains.
\begin{lemma}
\label{lem:decom}
Let  $N\in\mathbb{N}\cap[2,\infty)$ and  $\Omega\subset\mathbb R^N$ be a bounded Lipschitz domain.
Let $m\in\mathbb{N}$.
Then there exist a constant
$\ell_0\in (0,1)$,
depending only on  $\Omega$,
and a finite family of  sets $\{U_{m,\alpha}\}_{\alpha=1}^{J_m}$ such that the following statements hold:
\begin{itemize}
\item  [$\mathrm{(i)}$]
For any  given
$\alpha\in\{1,\ldots,J_m\}$, the set $U_{m,\alpha}$ is one of the following:
\begin{itemize}
\item  [$\mathrm{(i)_1}$]
 an axis-parallel  cube of edge length $\ell_m:=2^{-m}\ell_0$ contained in $\Omega$;
\item  [$\mathrm{(i)_2}$]
there exist  a
rigid motion $T_\alpha:\mathbb R^N\to\mathbb R^N$,
an axis-parallel cube $Q'_\alpha\subset\mathbb R^{N-1}$ of edge length $\ell_m:=2^{-m}\ell_0$, and
a Lipschitz function $\Gamma_\alpha:Q'_\alpha\to\mathbb R$ whose Lipschitz constant is
bounded by the corresponding one of $\Omega$.
Let
\begin{align*}
H_\alpha:=Q'_\alpha\times\Bigl(0,\frac{\ell_m}{2}\Bigr)
\qquad\text{and}\qquad
\Phi_\alpha(x',t):=(x',\,t+\Gamma_\alpha(x')).
\end{align*}
Then $U_{m,\alpha}=T_\alpha^{-1}\bigl(\Phi_\alpha(H_\alpha)\bigr)$.
\end{itemize}
\item  [$\mathrm{(ii)}$]
\begin{align*}
\Omega\subset \bigcup_{\alpha=1}^{J_m} U_{m,\alpha}.
\end{align*}
\item  [$\mathrm{(iii)}$]
For any $\alpha\in\{1,\ldots,J_m\}$,
$U_{m,\alpha}\subset \Omega,$
$\operatorname{diam}(U_{m,\alpha})\leq C_{\mathrm{diam}}\,\ell_m,$
and $\frac{1}{2}\ell_m^N\leq |U_{m,\alpha}|\leq\ell_m^N,$
where  $C_{\mathrm{diam}}$ is a  positive constant depending only on $\Omega.$
\item  [$\mathrm{(iv)}$]
\begin{align*}
\sum_{\alpha=1}^{J_m}\mathbf 1_{U_{m,\alpha}}\le N_0,
\end{align*}
where $N_0$ depends  only on $\Omega.$
\item  [$\mathrm{(v)}$]
There exists a   finite family of measurable sets  $\{E_{m,\alpha}\}_{\alpha=1}^{J_m}$ such that, for any $\alpha\in\{1,\ldots,J_m\}$, $E_{m,\alpha}\subset U_{m,\alpha}$, $\{E_{m,\alpha}\}_{\alpha=1}^{J_m}$ is  pairwise disjoint, and
\begin{align*}
\Omega= \bigcup_{\alpha=1}^{J_m}E_{m,\alpha}.
\end{align*}
\end{itemize}
\end{lemma}
\begin{proof}
Let $x_0\in\partial \Omega.$
By the definition of bounded Lipschitz domains, we  find that
there exist $r_{x_0}\in (0,\infty)$,
a rigid motion $T_{x_0}:\mathbb R^N\to\mathbb R^N$, and
a Lipschitz function
$\gamma_{x_0}:\ (-r_{x_0},r_{x_0})^{N-1}\to\mathbb R$
such that  $T_{x_0}(x_0)=0,$
\begin{align*}
T_{x_0}(\Omega\cap T_{x_0}^{-1}(Q_{x_0}))
=
\bigl\{(x',x_N)\in Q_{x_0}:\ x_N>\gamma_{x_0}(x')\bigr\},
\end{align*}
and
\begin{align*}
T_{x_0}(\partial\Omega\cap T_{x_0}^{-1}(Q_{x_0}))
=
\bigl\{(x',x_N)\in Q_{x_0}:\ x_N=\gamma_{x_0}(x')\bigr\},
\end{align*}
where
\begin{align*}
Q'_{x_0}:=\Bigl(-r_{x_0},r_{x_0}\Bigr)^{N-1}
\ \ \text{and}\ \
Q_{x_0}:=Q'_{x_0}\times(-r_{x_0},r_{x_0}).
\end{align*}
Obviously,  there exists $\delta_{x_0}\in (0,1)$ such that, for any $x'\in \delta_{x_0}Q'_{x_0},$
\begin{align*}
|\gamma_{x_0}(x')|
\leq
\mathrm{Lip}(\gamma_{x_0};Q'_{x_0})|x'|
\leq
\mathrm{Lip}(\gamma_{x_0};Q'_{x_0})\sqrt{N-1}\delta_{x_0} r_{x_0}
<\frac{r_{x_0}}{8}.
\end{align*}
Since
the family of open balls
$\{B(x_0,\frac{\delta_{x_0}r_{x_0}}{16}):\ x_0\in\partial\Omega\}$
is an open cover of $\partial\Omega$ and
$\partial\Omega$ is compact,
it follows that
there exist finitely many boundary points $x^{(1)},\dots,x^{(L)}\in\partial\Omega$
such that
\[
\partial\Omega\subset \bigcup_{\lambda=1}^L B\left(x^{(\lambda)},\frac{\delta_{\lambda}r_{\lambda}}{16}\right),
\]
where, for simplicity of notation, let $r_\lambda:=r_{x^{(\lambda)}},$
$T_\lambda:=T_{x^{(\lambda)}}$, $\gamma_\lambda:=\gamma_{x^{(\lambda)}},$ $Q'_\lambda:=(-r_\lambda,r_\lambda)^{N-1}$,  and $Q_\lambda:=Q'_\lambda\times(-r_\lambda,r_\lambda).$
Then
\begin{equation}\label{eq:chart-lambda}
T_\lambda(\Omega)\cap Q_\lambda
=
\{(x',x_N)\in Q_\lambda:\ x_N>\gamma_\lambda(x')\}.
\end{equation}
and, for any $x'\in \delta_\lambda Q'_\lambda,$
\begin{equation}\label{eq:small-height-lambda}
|\gamma_\lambda(x')| <
\frac{r_\lambda}{8}.
\end{equation}
Let
\begin{align*}
L_*:=\max_{1\le\lambda\le L}\operatorname{Lip}(\gamma_\lambda;Q'_\lambda),
\
r_*:=\min_{1\le\lambda\le L} r_\lambda,\ \text{and}\
\delta:=\min_{1\leq\lambda\leq L}\delta_{\lambda}.
\end{align*}
Let
\begin{align}\label{eq:M-choice}
M\in\mathbb{N}\cap(2(1+L_\ast)\sqrt{N},\infty)\ \ \text{and}\ \ \ell_0\in\left(0,\min\left\{1,\frac{\delta r_*}{16\sqrt N},\frac{\delta r_*}{8M}\right\}\right).
\end{align}
Let
\begin{align*}
\Omega_m^{\mathrm{int}}
:=
\left\{x\in\Omega:\  \mathop{\mathrm{dist}\,}(x,\partial\Omega)>\sqrt N\,\ell_m\right\}
\end{align*}
and
\begin{align*}
\Omega_m^{\mathrm{bd}}
:=
\Omega\setminus\Omega_m^{\mathrm{int}}
=
\left\{x\in\Omega:\  \mathop{\mathrm{dist}\,}(x,\partial\Omega)\le \sqrt N\,\ell_m\right\}.
\end{align*}

\textbf{Step 1: Cover of $\Omega_m^{\mathrm{int}}$.}
For any $\sigma=(\sigma_1,\dots,\sigma_N)\in\{0,\frac{1}{2}\}^N$
and  $k=(k_1,\dots,k_N)\in\mathbb Z^N$, define
\[
Q_{m,k}^{\sigma}
:=
\ell_m(k+\sigma)+(0,\ell_m)^N.
\]
Then it is easy to verify that
\begin{align}\label{2.5x}
\mathbb{R}^{N}=\bigcup_{\sigma,k}Q^\sigma_{m,k}\ \ \text{and}\ \ \sum_{\sigma,k}\mathbf{1}_{Q^\sigma_{m,k}}\leq 2^N.
\end{align}
Let $\{U_{m,\alpha}^{\mathrm{int}}\}_{\alpha\in I_m^{\mathrm{int}}}$ denote the subfamily of all cubes $Q_{m,k}^{\sigma}$
such that $Q_{m,k}^{\sigma}\cap \Omega_m^{\mathrm{int}}
\neq\emptyset.$
Obviously,
$\Omega_m^{\mathrm{int}}\subset
\bigcup_{\alpha\in I_m^{\mathrm{int}}}U_{m,\alpha}^{\mathrm{int}}.$
Since $\Omega$ is bounded,
from \eqref{2.5x}, we deduce  that $I_m^{\mathrm{int}}$ is finite.
Moreover, let $x\in U_{m,\alpha}^{\mathrm{int}}\cap \Omega_m^{\mathrm{int}}$
with $\alpha\in I_m^{\mathrm{int}}.$
Then, for any $y\in U_{m,\alpha}^{\mathrm{int}},$ we have
\begin{align*}
|x-y|\le \mathrm{diam}\,(U_{m,\alpha}^{\mathrm{int}})=\sqrt N\,\ell_m< \mathop{\mathrm{dist}\,}(x,\partial\Omega).
\end{align*}
Thus, for any $\alpha\in I_m^{\mathrm{int}}$, $U_{m,\alpha}^{\mathrm{int}}
\subset\Omega,$ which is also the
desired property.

\textbf{Step 2: Cover of $\Omega_m^{\mathrm{bd}}$.}
Let  $\lambda\in\{1,\dots,L\}$ and
\[
\widetilde Q'_\lambda:=\left(-\frac{\delta_\lambda r_\lambda}{8},\frac{\delta_\lambda r_\lambda}{8}\right)^{N-1}\subset\mathbb R^{N-1}.
\]
For any
\[
\sigma'=(\sigma_1,\dots,\sigma_{N-1})\in\left\{0,\frac{1}{2}\right\}^{N-1}
\]
and  $k'\in\mathbb Z^{N-1}$, let
\[
Q_{m,k'}^{\sigma'}
:=
\ell_m(k'+\sigma')+(0,\ell_m)^{N-1}
\subset\mathbb R^{N-1}.
\]
Then
\begin{align}\label{houhui}
\mathbb{R}^{N-1}=\bigcup_{\sigma',k'}Q^{\sigma'}_{m,k'}\ \ \text{and}\ \ \sum_{\sigma',k'}\mathbf{1}_{Q^{\sigma'}_{m,k'}}\leq 2^{N-1}.
\end{align}
Let $\{Q'_{\lambda,\beta}\}_{\beta\in I_{m,\lambda}}$ be the subfamily of all cubes satisfying $Q_{m,k'}^{\sigma'}\cap \widetilde Q'_\lambda\neq\emptyset.$
Obviously,
$$\widetilde Q'_\lambda\subset \bigcup_{\beta\in I_{m,\lambda}}Q'_{\lambda,\beta}.$$
From \eqref{houhui}, we further infer that
$I_{m,\lambda}$ is finite. Moreover,  let
$y'\in Q'_{\lambda,\beta}\cap \widetilde Q'_\lambda$ with $\beta\in I_{m,\lambda}$.
Then, for any $x'\in Q'_{\lambda,\beta}$, we have
\begin{align*}
|x'|_\infty
\leq
|x'-y'|_\infty+|y'|_\infty
\leq
\ell_m+\frac{\delta_\lambda r_\lambda}{8}
<
\delta_\lambda r_\lambda.
\end{align*}
Thus, for any $\beta\in I_{m,\lambda}$, $Q'_{\lambda,\beta}\subset \delta_\lambda Q'_\lambda.$

For any $j\in\{1,\ldots,M-1\}$ and $\tau\in\{0,1\}$, let
\begin{align*}
c_{m,j}^{(0)}:=j\frac{\ell_m}{2},\,
c_{m,j}^{(1)}:=\left(j+\frac{1}{2}\right)\frac{\ell_m}{2},\ \ \text{and}\ \ I_{m,j}^{(\tau)}=c_{m,j}^{(\tau)}+\left(0,\frac{\ell_m}{2}\right).
\end{align*}
Then
\begin{equation}\label{eq:layer-cover}
\left(0,M\frac{\ell_m}{2}\right)\subset
\left[\bigcup_{j=0}^{M-1} I_{m,j}^{(0)}\right]
\cup
\left[\bigcup_{j=0}^{M-1} I_{m,j}^{(1)}\right]
\end{equation}
and
\begin{align}\label{houhui2}
\sum_{j=0}^{M-1}\mathbf{1}_{I^{(0)}_{m,j}}+\mathbf{1}_{I^{(1)}_{m,j}}
\leq2.
\end{align}
For any $\beta\in I_{m,\lambda},$ let
\[
H_{\lambda,\beta}:=
Q'_{\lambda,\beta}\times\left(0,\frac{\ell_m}{2}\right).
\]
For any $\tau\in\{0,1\}$,
$\beta\in I_{m,\lambda},$ $j\in\{0,1,\ldots,M-1\}$, and $(x',t)\in H_{\lambda,\beta}$, let
\begin{align*}
\Phi_{\lambda,\beta,j}^{(\tau)}(x',t)
:=
\left(x',\,t+\gamma_\lambda(x')+c_{m,j}^{(\tau)}\right)\ \ \text{and}\ \ U_{m,\lambda,\beta,j}^{(\tau)}
:=
T_\lambda^{-1}\left(\Phi_{\lambda,\beta,j}^{(\tau)}(H_{\lambda,\beta})\right).
\end{align*}
Now, we show that
\begin{align}\label{ljdx2}
\Omega_m^{\mathrm{bd}}\subset
\bigcup_{\tau,\lambda,\beta,j} U_{m,\lambda,\beta,j}^{(\tau)}.
\end{align}
Let $x\in\Omega_m^{\mathrm{bd}}$.
Then there exists $\xi\in\partial \Omega$ such that
\begin{align*}
|x-\xi|=\mathop{\mathrm{dist}\,}(x,\partial\Omega)\le \sqrt N\,\ell_m.
\end{align*}
Since the balls $\{B(x^{(\lambda)},\frac{\delta_\lambda r_{\lambda}}{16})\}_{\lambda=1}^L$ cover $\partial\Omega$,  it follows that there exists
$\lambda\in\{1,\dots,L\}$ such that
$\xi\in B(x^{(\lambda)},\frac{\delta_\lambda r_{\lambda}}{16}).$
Thus,
\[
|x-x^{(\lambda)}|
\le
|x-\xi|+|\xi-x^{(\lambda)}|
<
\sqrt{N}\ell_m+\frac{\delta_\lambda r_\lambda}{16}
\leq
\sqrt{N}\ell_0+\frac{ \delta_\lambda r_\lambda}{16}
<\frac{\delta_\lambda r_\lambda}{8},
\]
where we used \eqref{eq:M-choice} in the last step.
This implies that
\begin{align}\label{ljdx}
T_\lambda(x)\in B\left(\mathbf{0},\frac{\delta_\lambda r_\lambda}{8}\right)
\subset Q_\lambda.
\end{align}
Thus,
\begin{align*}
T_\lambda(x)=(x',x_N)\ \ \text{and}\ \ x'\in \left(-\frac{\delta_\lambda r_\lambda}{8},\frac{\delta_\lambda r_\lambda}{8}\right)^{N-1}=\widetilde Q'_\lambda.
\end{align*}
By \eqref{houhui}, we find that there exists $\beta\in I_{m,\lambda}$ such that $x'\in Q^{'}_{\lambda,\beta}.$
Let $t_x:=x_N-\gamma_\lambda(x').$
Using the fact that $\xi\in B(x^{(\lambda)},\frac{\delta_\lambda r_{\lambda}}{16})\subset T^{-1}_\lambda(Q_\lambda)$, we  conclude that
$T_\lambda(\xi)=(\xi',\gamma_\lambda(\xi'))$
for some $\xi'\in Q'_\lambda$.
By this and \eqref{ljdx}, we obtain
\begin{align*}
0<t_x
&=
\bigl(x_N-\gamma_\lambda(\xi')\bigr)
+
\bigl(\gamma_\lambda(\xi')-\gamma_\lambda(x')\bigr)
\leq
|x_N-\gamma_\lambda(\xi')|+L_*|x'-\xi'|\\
&\leq
(1+L_\ast)|x-\xi|
\leq
(1+L_\ast)\sqrt{N}\ell_m
<\frac{M\ell_m}{2},
\end{align*}
where we used \eqref{eq:M-choice} in the last step.
This, combined with \eqref{eq:layer-cover},  implies that there exist $j\in\{0,\dots,M-1\}$,
$\tau\in\{0,1\}$, and $t\in (0,\frac{\ell_m}{2})$ such that
$t_x=t+c_{m,j}^{(\tau)}\in I_{m,j}^{(\tau)}.$
From this and the fact that $(x',t)\in H_{\lambda,\beta}$,  we deduce  that
\[
T_\lambda(x)
=
(x',x_N)
=
\bigl(x',\,\gamma_\lambda(x')+t_x\bigr)
=
\bigl(x',\,\gamma_\lambda(x')+c_{m,j}^{(\tau)}+t\bigr)
=
\Phi_{\lambda,\beta,j}^{(\tau)}(x',t).
\]
and hence $x\in U_{m,\lambda,\beta,j}^{(\tau)}.$
This proves \eqref{ljdx2}.

\textbf{Step 3: Verification of the remaining properties.}
Let $\tau\in\{0,1\},$ $\lambda\in\{1,\ldots,L\},$
$j\in\{0,1,\ldots,M-1\},$ and $\beta\in I_{m,\lambda}.$
Next, we show that
\begin{align*}
U_{m,\lambda,\beta,j}^{(\tau)}\subset \Omega.
\end{align*}
To this end, we only need to prove that
\begin{align}\label{xiaguang}
\Phi_{\lambda,\beta,j}^{(\tau)}(H_{\lambda,\beta})
\subset
T_{\lambda}\left(\Omega\cap T^{-1}_\lambda(Q_\lambda)\right).
\end{align}
Let $(x',t)\in H_{\lambda,\beta}$ and
\begin{align*}
x:=\Phi_{\lambda,\beta,j}^{(\tau)}(x',t)
=
\left(x',\gamma_\lambda(x')+c_{m,j}^{(\tau)}+t\right).
\end{align*}
By the fact that $x'\in Q'_{\lambda,\beta}\subset \delta_\lambda Q'_\lambda$ and \eqref{eq:small-height-lambda}, we obtain
\begin{align*}
x_N=\gamma_\lambda(x')+c_{m,j}^{(\tau)}+t>\gamma_\lambda(x')>-\frac{r_\lambda}{8}>-r_\lambda
\end{align*}
and
\begin{align*}
x_N&=\gamma_\lambda(x')+c_{m,j}^{(\tau)}+t
<\gamma_\lambda(x')+\left(M-1+\frac{1}{2}\right)\frac{\ell_m}{2}+\frac{\ell_m}{2}\\
&\leq
\frac{r_\lambda}{8}+\left(M+\frac{1}{2}\right)\frac{\ell_0}{2}
<\frac{r_\lambda}{8}+\left(M+\frac{1}{2}\right)\frac{r_\lambda}{16M}<r_\lambda.
\end{align*}
Using this and  \eqref{eq:chart-lambda}, we conclude that \eqref{xiaguang} holds.

Since rigid motions preserve distances and the Lebesgue measure,
it follows that
\begin{align*}
|U_{m,\lambda,\beta,j}^{(\tau)}|
=|\Phi_{\lambda,\beta,j}^{(\tau)}(H_{\lambda,\beta})|
=
|Q'_{\lambda,\beta}|\cdot \frac{\ell_m}{2}
=
\ell_m^{N-1}\cdot \frac{\ell_m}{2}
=
\frac12\,\ell_m^N.
\end{align*}
Let
\begin{align*}
x:=\Phi_{\lambda,\beta,j}^{(\tau)}(x',t)
\ \ \text{and}\ \
y:=\Phi_{\lambda,\beta,j}^{(\tau)}(y',s)
\end{align*}
with $(x',t),(y',s)\in H_{\lambda,\beta}$.
Then
\begin{align*}
|x-y|
&=
\bigl|(x'-y',\,t-s+\gamma_\lambda(x')-\gamma_\lambda(y'))\bigr|
\leq
|x'-y'|+|t-s|+|\gamma_\lambda(x')-\gamma_\lambda(y')| \\
&\leq
(1+L_*)|x'-y'|+|t-s|
\leq
(1+L_\ast)\sqrt{N-1}\ell_m+\frac{\ell_m}{2}.
\end{align*}
This  implies that
\begin{align*}
\operatorname{diam}\left(U_{m,\lambda,\beta,j}^{(\tau)}\right)
\le
\left((1+L_*)\sqrt{N-1}+\frac12\right)\ell_m.
\end{align*}
By \eqref{houhui} and \eqref{houhui2}, we have
\begin{align*}
\sum_{\lambda=1}^L\sum_{j=0}^{M-1}\sum_{\beta\in I_{m,\lambda}}\sum_{\tau\in\{0,1\}}\mathbf{1}_{U_{m,\lambda,\beta,j}^{(\tau)}}
\leq L2^{N-1}2=L2^N.
\end{align*}

Combining the covers of $\Omega^{\mathrm{int}}_m$ and $\Omega^{\mathrm{bd}}_m$, we obtain the cover $\{U_{m,\alpha}\}_{\alpha=1}^{J_m}$ of $\Omega$.
Finally, for any $\alpha\in \{2,\dots,J_m\}$,  let
\begin{align*}
E_{m,1}:=U_{m,1}\cap\Omega
\end{align*}
and
\begin{align*}
E_{m,\alpha}
:=
(U_{m,\alpha}\cap\Omega)\setminus \bigcup_{\beta=1}^{\alpha-1} E_{m,\beta}.
\end{align*}
Then $\{E_{m,\alpha}\}_{\alpha=1}^{J_m}$ satisfies (v). This finishes the proof of Lemma \ref{lem:decom}.
\end{proof}
\subsection{Weighted  Fractional Poincar\'e Inequality On Cubes}\label{sec2.3}
In  this subsection, we show the
weighted  fractional Poincar\'e inequality on cubes.
Let the \emph{symbol}
$\mathcal{Q}_{\mathrm{ax}}$ denote the family of all cubes in $\mathbb{R}^{N}$
whose edges are parallel with the coordinate axes and
let the \emph{symbol} $\mathcal{Q}_{\mathrm{all}}$ denote the family of all cubes with  arbitrary orientation in $\mathbb{R}^{N}$. For any given $Q\in \mathcal{Q}_{\mathrm{ax}},$
the \emph{symbol} $\mathcal{D}(Q)$
denotes the collection of all dyadic subcubes of $Q.$
\begin{definition}
Let $p,s\in (0,\infty)$ and
$w\in L^1_{\mathrm{loc}}$ with $w\geq 0$, and
let $Q\in\mathcal{Q}_{\mathrm{ax}}$ and
$a:\ \mathcal{D}(Q)\to [0,\infty)$.
We say $a\in SD_p^s(w;Q)$
if there exists a positive constant $C$ such that,
for any
family of disjoint dyadic subcubes
$\{P_i\}_{i\in I}\subset \mathcal{D}(Q)$,
\begin{align*}
\left[
\sum_{i\in I} a(P_i)^p\,\frac{w(P_i)}{w(Q)}
\right]^{\frac{1}{p}}
\leq
C
\left(
\frac{\big|\bigcup_{i\in I} P_i\big|}{|Q|}
\right)^{\frac{1}{s}}
a(Q).
\end{align*}
The best possible constant $C$,
namely  the infimum of the constants in the last inequality, is
denoted by $a_{SD_p^s(w;Q)}$.
\end{definition}
By an argument  similar to that used in the proof of \cite[Theorem 5.3]{llo22}, we have the  following conclusion; we omit the details here.
\begin{proposition}\label{self}
Let $p,s\in[1,\infty)$,
$w\in L^1_{\mathrm{loc}}$ with $w\geq 0$,
$Q\in\mathcal{Q}_{\mathrm{ax}}$,
and $a\in SD_p^s(w;Q)$.
Let
$f\in L^1_{\mathrm{loc}}$ be such that, for any $P\in\mathcal{D}(Q),$
\begin{align*}
\frac1{|P|}\int_P |f-f_P|
\leq a(P).
\end{align*}
Then there exists a positive constant $C=C_{N}$ such that
\begin{align*}
\left[
\frac1{w(Q)}\int_{Q}|f(x)-f_{Q}|^p\,dw(x)
\right]^{\frac{1}{p}}
\le
Csa_{SD_p^s(w;Q)}a(Q).
\end{align*}
\end{proposition}
By Proposition \ref{self}, we have the following lemma.
\begin{lemma}\label{poincare1}
Let $p,q\in [1,\infty)$ and  $\delta\in(\frac12,1)$. Then there exists a  positive constant
$C=C_{N,p,q}$ such that, for any  cube $Q\in\mathcal{Q}_{\mathrm{ax}}$ and
 any $u,w\in L^1_{\mathrm{loc}}$ with $w\geq 0,$
\begin{align*}
&\left[\int_Q |u(x)-u_Q|^p\,w(x)\,dx\right]^{\frac{1}{p}}\\
&\quad\leq
C(1-\delta)^{\frac{1}{q}}\ell(Q)^\delta
\left\{
\int_Q
\left[
\int_Q \frac{|u(x)-u(y)|^q}{|x-y|^{N+\delta q}}\,dy
\right]^{\frac{p}{q}}
\mathcal{M}(w\mathbf{1}_Q)(x)\,dx
\right\}^{\frac{1}{p}}.
\end{align*}
\end{lemma}
\begin{proof}
Let $Q\in\mathcal{Q}_{\mathrm{ax}}$ and $u,w\in L^1_{\mathrm{loc}}$ satisfy $w\geq 0.$
We first show that there exists a positive constant $C$, depending only on $N,$ $p,$ and $q$, such that, for any $P\in\mathcal{D}(Q),$
\begin{equation}\label{eq:seed-mixed}
\frac1{|P|}\int_P |u(x)-u_P|\,dx
\leq
C(1-\delta)^{\frac{1}{q}}\ell(P)^\delta
\left\{
\frac1{|P|}\int_P \left[\mathcal{D}^{\delta}_{q,P}u(x)\right]^p\,dx
\right\}^{\frac{1}{p}}.
\end{equation}
Indeed, if $q\in [1,p],$
by \cite[(1.8)]{hmpv23} and H\"older's  inequality, we have
\begin{align*}
\frac1{|P|}\int_P |u(x)-u_P|\,dx
&\lesssim
(1-\delta)^{\frac{1}{q}}\ell(P)^\delta
\left\{
\frac1{|P|}\int_P \left[\mathcal{D}^{\delta}_{q,P}u(x)\right]^q\,dx
\right\}^{\frac{1}{q}}\\
&\leq
(1-\delta)^{\frac{1}{q}}\ell(P)^\delta
\left\{
\frac1{|P|}\int_P \left[\mathcal{D}^{\delta}_{q,P}u(x)\right]^p\,dx
\right\}^{\frac{1}{p}}.
\end{align*}
Thus, \eqref{eq:seed-mixed} holds in this case. Let $q\in (p,\infty)$ and $\sigma:=2\delta-1\in (0,1).$
Using \cite[(1.8)]{hmpv23}, we find that
\begin{align}\label{zsydsb}
\frac1{|P|}\int_P |u(x)-u_P|\,dx
\lesssim
(1-\sigma)^{\frac{1}{q}}\ell(P)^\sigma
\left\{
\frac1{|P|}\int_P \left[\mathcal{D}^{\sigma}_{q,P}u(x)\right]^p\,dx
\right\}^{\frac{1}{p}}.
\end{align}
On the other hand, we have, for any $x\in P,$
\begin{align*}
\left[\mathcal{D}^\sigma_{q,P}u(x)\right]^q
=\int_{P}\frac{|u(x)-u(y)|^q |x-y|^{q(\delta-\sigma)}}{|x-y|^{N+\delta q}}\,dy
\lesssim
\ell(P)^{q(1-\delta)}\left[\mathcal{D}^\delta_{q,P}u(x)\right]^q.
\end{align*}
From this and \eqref{zsydsb}, it   follows  that \eqref{eq:seed-mixed} holds in this case.

Observe that, for any $P\in\mathcal{D}(Q)$ and $x\in P,$
\begin{align*}
\frac{w(P)}{|P|}
\leq
\mathcal{M}(w\mathbf{1}_{Q})(x)
\end{align*}
By this and \eqref{eq:seed-mixed}, we conclude that, for any $P\in\mathcal{D}(Q),$
\begin{align*}
\frac1{|P|}\int_P |u-u_P|
\leq
C(1-\delta)^{\frac{1}{q}}\ell(P)^\delta
\left[
\frac1{w(P)}\int_P \left[\mathcal{D}^{\delta}_{q,P}u(x)\right]^p \mathcal{M}(w\mathbf{1}_Q)(x)\,dx
\right]^{\frac{1}{p}}
=:a(P).
\end{align*}
Consequently,
using this, we obtain,
for any
family of disjoint dyadic subcubes $\{P_i\}_{i\in I}\subset \mathcal{D}(Q)$,
\begin{align*}
\sum_{i\in I} a(P_i)^p\frac{w(P_i)}{w(Q)}
&=
\frac{C^p(1-\delta)^{\frac{p}{q}}}{w(Q)}
\sum_{i\in I} \ell(P_i)^{\delta p}
\int_{P_i} \left[\mathcal D^{\delta}_{q,P_i}u(x)\right]^p \mathcal{M}(w\mathbf{1}_Q)(x)\,dx \\
&\leq
\frac{C^p(1-\delta)^{\frac{p}{q}}}{w(Q)}
\sup_{i\in I} \ell(P_i)^{\delta p}
\int_{Q} \left[\mathcal D_{q,Q}^\delta u(x)\right]^p \mathcal{M}(w\mathbf{1}_Q)(x)\,dx\\
&\leq
\left(
\frac{\left|\bigcup_{i\in I} P_i\right|}{|Q|}
\right)^{\frac{\delta p}{N}}
a(Q)^p.
\end{align*}
This shows that $a\in SD_p^{\frac{N}{\delta}}(w;Q)$ and  $a_{SD_p^{\frac{N}{\delta}}(w;Q)}\leq 1.$
Applying this and
Proposition \ref{self}, we obtain the desired conclusion. This finishes the proof of Lemma \ref{poincare1}.
\end{proof}
\begin{remark}
We point out that, in the case $p=q\in [1,\infty)$,
Lemma \ref{poincare1} coincides with \cite[Theorem 2.9(b)]{hmpv23}.
\end{remark}
\begin{lemma}\label{poincare2}
Let $p,q\in [1,\infty)$ and  $\delta\in(\frac{1}{2},1)$.
Let $Q'\subset\mathbb R^{N-1}$ be an axis-parallel cube of edge length $\ell$, and let
$H:=Q'\times(0,\frac{\ell}{2})$.  Then there exists a  positive constant
$C=C_{N,p,q}$ such that,
for any $u,w\in L^1_{\mathrm{loc}}$ with $w\geq 0,$
\begin{align}\label{xiong0}
&\left[\int_H |u(x)-u_H|^pw(x)\,dx\right]^{\frac{1}{p}}\notag\\
&\quad\leq
C(1-\delta)^{\frac{1}{q}}\ell^\delta
\left\{
\int_H
\left[
\int_H \frac{|u(x)-u(y)|^q}{|x-y|^{N+\delta q}}\,dy
\right]^{\frac{p}{q}}
\mathcal{M}(w\mathbf{1}_H)(x)\,dx
\right\}^{\frac{1}{p}}.
\end{align}
\end{lemma}
\begin{proof}
Let $H^{-}:=Q'\times(-\frac{\ell}{2},0)$, $\Pi:=Q'\times\{0\},$
and $Q:=H\cup H^{-}\cup \Pi.$
 For any $(x',x_N)\in Q,$ let
$\sigma(x',x_N):=(x',-x_N).$
For any given $u,w\in L^1_{\mathrm{loc}}$ with $w\geq 0,$ define
\[
\widetilde{u}(x):=
\begin{cases}
u(x),&x\in H,\\
u(\sigma x),&x\in H^-\\
0,&x\in \Pi,
\end{cases}
\ \ \text{and}\ \
\widetilde{w}(x):=
\begin{cases}
w(x),&x\in H,\\
w(\sigma x),&x\in H^-,\\
0,&x\in \mathbb R^N\setminus(H\cup H^-).
\end{cases}
\]
It is  obvious that
$(\widetilde{u})_Q=u_H$
and
\[
\int_Q |\widetilde{u}(x)-(\widetilde{u})_Q|^p\widetilde{w}(x)\,dx
=
2\int_H |u(x)-u_H|^pw(x)\,dx.
\]
Using this and
applying Lemma \ref{poincare1}, we obtain
\begin{equation}\label{xiong}
\left[\int_H |u(x)-u_H|^pw(x)\,dx\right]^{\frac{1}{p}}
\lesssim
(1-\delta)^{\frac{1}{q}}\ell^\delta
\left\{
\int_Q \left[\mathcal D_{q,Q}^{\delta}\widetilde{u}(x)\right]^p \mathcal{M}(\widetilde{w}\mathbf{1}_Q)(x)\,dx
\right\}^{\frac{1}{p}}.
\end{equation}
It is easy to verify that
\begin{align*}
\int_Q \left[\mathcal D_{q,Q}^{\delta}\widetilde{u}(x)\right]^p \mathcal{M}(\widetilde{w}\mathbf{1}_Q)(x)\,dx
\lesssim
\int_H \left[\mathcal D_{q,H}^{\delta}u(x)\right]^p \mathcal{M}(w\mathbf{1}_H)(x)\,dx.
\end{align*}
This, combined with \eqref{xiong}, implies that \eqref{xiong0} holds. This finishes the proof of Lemma \ref{poincare2}.
\end{proof}
\begin{lemma}\label{poincare3}
Let  the  notation be the same as in Lemma \ref{poincare2}.
Let  $\gamma:Q'\to\mathbb R$ be a Lipschitz  function with Lipschitz constant $L$
and $T:\ \mathbb{R}^{N}\to\mathbb{R}^{N}$ be a rigid motion.
For any $(x',t)\in H$, let
\begin{align*}
\Phi(x',t):=(x',t+\gamma(x'))\ \ \text{and}\ \  U:=T(\Phi(H)).
\end{align*}
Then there exists a  positive constant
$C=C_{N,p,q,L}$ such that,
for any $u,w\in L^1_{\mathrm{loc}}$ with $w\geq 0,$
\begin{align*}
&\left[\int_U |u(x)-u_U|^pw(x)\,dx\right]^{\frac{1}{p}}\notag\\
&\quad\leq
C(1-\delta)^{\frac{1}{q}}\ell^\delta
\left\{
\int_U
\left[
\int_U \frac{|u(x)-u(y)|^q}{|x-y|^{N+\delta q}}\,dy
\right]^{\frac{p}{q}}
\mathcal{M}(w\mathbf{1}_U)(x)\,dx
\right\}^{\frac{1}{p}}.
\end{align*}
\end{lemma}
\begin{proof}
Let $\Psi:=T\circ \Phi.$
Then $U=\Psi(H)$.
Since $T$ is a rigid motion,
we infer that there exist an orthogonal matrix $A$ and
$b\in\mathbb R^N$ such that $Tz=Az+b$ for any  $z\in\mathbb{R}^{N}$. Moreover, by Rademacher's theorem,
$\gamma$ is differentiable almost everywhere.
Using this, we conclude that, for almost every $(x',t)\in H,$
\begin{align*}
D\Phi(x',t)
=
\begin{pmatrix}
I_{N-1} & 0\\
\nabla\gamma(x') & 1
\end{pmatrix}
\end{align*}
and hence $\det D\Phi(x',t)=1$,
where $D\Phi$ denotes the Jacobian matrix of $\Phi.$
This implies that, for almost every $x\in H,$  $D\Psi(x)=A D\Phi(x)$ and hence
\begin{align}\label{hiker}
|\det D\Psi(x)|=|\det A|\,|\det D\Phi(x)|=1.
\end{align}
From a straightforward computation, we  deduce that, for any $x,y\in H,$
\begin{align*}
(1+L)^{-1}|x-y|
\leq
|\Phi(x)-\Phi(y)|
\leq
(1+L)|x-y|
\end{align*}
and
\begin{align}\label{youshigan}
(1+L)^{-1}|x-y|
\leq
|\Psi(x)-\Psi(y)|=|\Phi(x)-\Phi(y)|
\leq
(1+L)|x-y|.
\end{align}
For any $x\in \mathbb R^N$, let
\begin{align*}
v(x):=u(\Psi(x))\mathbf{1}_H(x)
\ \ \text{and}\ \
\widetilde w(x):=w(\Psi(x))\mathbf{1}_H(x).
\end{align*}
By the change-of-variables formula and \eqref{hiker}, we find that  $|U|=|H|$,
\begin{align*}
u_U
&=
\frac{1}{|U|}\int_U u(x)\,dx
=
\frac{1}{|H|}\int_H u(\Psi(x))\,dx
=
v_H,
\end{align*}
and
\begin{align*}
\int_U |u(x)-u_U|^p w(x)\,dx
=
\int_H |v(x)-v_H|^p\widetilde w(x)\,dx .
\end{align*}
Using this and applying Lemma \ref{poincare2} to $(H,v,\widetilde w)$, we conclude that
\begin{align}\label{meiqian23}
\int_U |u(x)-u_U|^p w(x)\,dx
\lesssim
(1-\delta)^{\frac{p}{q}}\ell^{\delta p}
\int_H
\left[\mathcal D^\delta_{q,H}v(x)\right]^p
\mathcal M(\widetilde w\mathbf{1}_H)(x)\,dx .
\end{align}
From \eqref{youshigan}, we infer that, for any $x\in H$,
\begin{align}\label{meiqian12}
\left[\mathcal D^\delta_{q,H}v(x)\right]^q
&=
\int_H
\frac{|u(\Psi(x))-u(\Psi(y))|^q}{|x-y|^{N+\delta q}}\,dy\notag\\
&\sim
\int_H
\frac{|u(\Psi(x))-u(\Psi(y))|^q}{|\Psi(x)-\Psi(y)|^{N+\delta q}}\,dy
\notag\\
&=\int_U
\frac{|u(\Psi(x))-u(z)|^q}{|\Psi(x)-z|^{N+\delta q}}\,dz
=
\left[\mathcal D^\delta_{q,U}u(\Psi(x))\right]^q.
\end{align}
Let $x\in H$ and  $B$ be a ball
in $\mathbb R^N$ such that $x\in B$. Denote by $r_B$ the radius of $B$.
By \eqref{youshigan},
for any $y\in B\cap H$,
\begin{align*}
|\Psi(y)-\Psi(x)|
\leq
(1+L)|y-x|
\leq
2(1+L)r_B.
\end{align*}
Thus,
\begin{align*}
\Psi(B\cap H)\subset B(\Psi(x),2(1+L)r_B).
\end{align*}
From this and  the change-of-variables formula, we deduce that
\begin{align*}
\frac{1}{|B|}\int_B \widetilde w(y)\mathbf{1}_H(y)\,dy
&=
\frac{1}{|B|}\int_{B\cap H} w(\Psi(y))\,dy=
\frac{1}{|B|}\int_{\Psi(B\cap H)} w(z)\mathbf{1}_U(z)\,dz
\\
&\lesssim
\frac{1}{|B|}
\int_{B(\Psi(x),2(1+L)r_B)} w(z)\mathbf{1}_U(z)\,dz
\lesssim
\mathcal{M}(w\mathbf{1}_U)(\Psi(x)).
\end{align*}
Taking the supremum over all balls $B$ containing $x$, we conclude that, for any $x\in H$,
\begin{align*}
\mathcal M(\widetilde w\mathbf{1}_H)(x)
\lesssim
\mathcal M(w\mathbf{1}_U)(\Psi(x)).
\end{align*}
This,  together with \eqref{meiqian12} and \eqref{meiqian23}, implies that
\begin{align*}
\int_U |u(x)-u_U|^p w(x)\,dx
&\lesssim
(1-\delta)^{\frac{p}{q}}\ell^{\delta p}
\int_H
\left[\mathcal D^\delta_{q,U}u(\Psi(x))\right]^p
\mathcal M(w\mathbf{1}_U)(\Psi(x))\,dx\\
&=
(1-\delta)^{\frac{p}{q}}\ell^{\delta p}
\int_U
\left[\mathcal D^\delta_{q,U}u(x)\right]^p
\mathcal M(w\mathbf{1}_U)(x)\,dx.
\end{align*}
This finishes the  proof of Lemma \ref{poincare3}.
\end{proof}
\subsection{Proof of Theorem \ref{thm-domain}}\label{sec2.4}
In this subsection, we prove Theorem \ref{thm-domain}.
We denote by $C^\infty_{\mathrm{c}}$ the space of all infinitely
differentiable functions on $\mathbb{R}^{N}$ with compact support.
We need the following lemma.
\begin{lemma}\label{lem:dir}
Let $e\in S^{N-1}$ and  $\varphi\in C_{\mathrm{c}}^\infty$.
For any $n\in\mathbb N$ and $x\in\mathbb{R}^N$, let
\begin{align*}
T_{n,e}\varphi(x):=\frac{2\sigma(S^{N-1})}{K(1,N)}
\int_{\{h\in\mathbb R^N:\ h\cdot e\ge0\}}
\frac{\varphi(x+h)-\varphi(x)}{|h|}\,\rho_n(|h|)\,dh,
\end{align*}
where $\rho_n$ is as in \eqref{tesuhe} and $K(1,N)$
as in \eqref{kchangshu}.
Then $T_{n,e}\varphi$  uniformly converges to  $e\cdot\nabla\varphi$
as $n\to\infty$.
\end{lemma}
\begin{proof}
By Taylor's expansion, we   find that, for any $x,h\in\mathbb{R}^{N},$
\[
\varphi(x+h)-\varphi(x)
=
\nabla\varphi(x)\cdot h + R_x(h),
\]
where
\begin{align*}
|R_x(h)|
\lesssim|h|^2.
\end{align*}
Then, for any $x\in\mathbb{R}^{N},$
\begin{align}\label{shenke1}
\frac{K(1,N)}{2\sigma(S^{N-1})}T_{n,e}\varphi(x)
&=
\int_{\{h\in\mathbb R^N:\ h\cdot e\ge0\}}
\frac{\nabla\varphi(x)\cdot h}{|h|}\rho_n(|h|)\,dh\notag\\
&\quad+
\int_{\{h\in\mathbb R^N:\ h\cdot e\ge0\}}
\frac{R_x(h)}{|h|}\rho_n(|h|)\,dh
\notag\\
&=:I_n(x)+J_n(x).
\end{align}
Choose an orthonormal basis
$\{e_1,\dots,e_N\}$ with $e_1=e.$
From this, it follows that, for any $x\in\mathbb{R}^{N},$
\begin{align*}
\nabla\varphi(x)=\sum_{j=1}^N a_j(x)e_j
\end{align*}
and
\begin{align*}
I_n(x)
&=\sum_{j=1}^N a_j(x)
\int_{\{h\cdot e_1\ge0\}}
\frac{h\cdot e_j}{|h|}\rho_n(|h|)\,dh\\
&=\sum_{j=1}^N a_j(x)\left(\int_0^\infty \rho_n(r)r^{N-1}\,dr\right)
\int_{\{\omega\cdot e_1\ge0\}} \omega\cdot e_j\,d\sigma(\omega)\\
&=\frac{a_1(x)}{\sigma(S^{N-1})}\int_{\{\omega\cdot e_1\ge0\}} \omega\cdot e_1\,d\sigma(\omega)=\frac{K(1,N)}{2\sigma(S^{N-1})}\,e\cdot\nabla\varphi(x).
\end{align*}
Moreover, for any $x\in\mathbb{R}^{N},$
\begin{align*}
\sup_{x\in\mathbb R^N}|J_n(x)|
\lesssim
\int_{\mathbb R^N}|h|\,\rho_n(|h|)\,dh
=
\frac{q(1-s_n)}{R^{q(1-s_n)}}\int_{0}^{R}r^{q(1-s_n)}\,dr\to0
\end{align*}
as $n\to\infty.$
This, together with \eqref{shenke1}, implies
the desired conclusion,
which completes the proof of Lemma \ref{lem:dir}.
\end{proof}
Now, we prove Theorem \ref{thm-domain}.
\begin{proof}[Proof of Theorem \ref{thm-domain}]
Let $m\in\mathbb{N}.$
By Lemma \ref{lem:decom}, we  find that there exists
a family of measurable sets $\{U_{m,\alpha}\}_{\alpha=1}^{J_m}$ having the  properties
in Lemma \ref{lem:decom}.
For any $n\in\mathbb{N}$, let
\[
E_m f_n:=\sum_{\alpha=1}^{J_m} (f_n)_{U_{m,\alpha}}\,\mathbf 1_{E_{m,\alpha}}.
\]
Then
\begin{align*}
\sup_{n\in\mathbb{N}}|(f_n)_{U_{m,\alpha}}|
\leq
\sup_{n\in\mathbb{N}}\frac{1}{|U_{m,\alpha}|}\int_{\Omega}|f_n(x)|\,dx
\leq
\frac{2}{\ell_m^N}\|\mathbf{1}_{\Omega}\|_{X'}\sup_{n\in\mathbb{N}}\|f_n\|_{X(\Omega)}
<\infty.
\end{align*}
By this, we have
\begin{align*}
\sup_{n\in\mathbb{N}}\|E_mf_n\|_{X(\Omega)}
\leq
\frac{2}{\ell_m^N}\|\mathbf{1}_{\Omega}\|_{X'}\sup_{n\in\mathbb{N}}\|f_n\|_{X(\Omega)}\sum_{\alpha=1}^{J_m}\|\mathbf{1}_{E_{m,\alpha}}\|_{X(\Omega)}
<\infty.
\end{align*}
This implies that $\{E_mf_n\}_{n\in\mathbb{N}}$ is relatively compact on $X(\Omega).$ Since $X(\Omega)$ is complete,  we infer that $\{E_mf_n\}_{n\in\mathbb{N}}$ is totally bounded on $X(\Omega).$

Next, we show that there  exists $n_m\in\mathbb{N}$ such that, for any $n\in\mathbb{N}\cap [n_m,\infty)$,
\begin{equation}\label{biaoqing}
\|f_n-E_mf_n\|_{X(\Omega)}
\lesssim
(1-s_n)^{\frac{1}{q}}\ell_m^{\frac{1}{2}}\left\|\mathcal{D}^{s_n}_{q,\Omega}(f_n)\right\|_{X(\Omega)}.
\end{equation}
Let $Y:=X^{\frac{1}{r}}.$
Using Lemma \ref{xduiou} and Proposition \ref{norm}, we obtain, for any $n\in\mathbb{N},$
\begin{align}\label{fanpan}
\|f_n-E_mf_n\|_{X(\Omega)}^r
=
\sup_{\genfrac{}{}{0pt}{}{g\geq0}{\|g\|_{Y'}=1}}
\int_\Omega |f_n(x)-E_mf_n(x)|^r g(x)\,dx.
\end{align}
Let $g\in Y'$ with $g\geq 0$ and $\|g\|_{Y'}=1.$ By Lemma \ref{lem:decom}(ii), we have
\begin{align}\label{shenglv}
\int_\Omega |f_n(x)-E_mf_n|^rg(x)\,dx
\leq
\sum_{\alpha=1}^{J_m}
\int_{U_{m,\alpha}} |f_n(x)-(f_n)_{U_{m,\alpha}}|^rg(x)\,dx.
\end{align}
Recall that, for any measurable set $U\subset \mathbb{R}^{N}$,
$s\in (0,1)$, and $f\in L^1_{\mathrm{loc}}$,
\begin{align*}
\mathcal{D}^{s}_{q,U}f(x):=
\left[\int_U
\frac{|f(x)-f(y)|^q}{|x-y|^{N+sq}}\,dy
\right]^{\frac{1}{q}}.
\end{align*}
Since $\lim_{n\to\infty}s_n=1,$
it follows that there  exists $n_m\in\mathbb{N}$ such that, for any $n\in\mathbb{N}\cap[n_m,\infty)$,
$s_n\in (\frac{1}{2},1).$
Using this and
Lemmas \ref{poincare1} and \ref{poincare3}, we  conclude that, for any $n\in\mathbb{N}\cap[n_m,\infty)$ and $\alpha\in\{1,\ldots,J_m\}$,
\begin{align*}
\int_{U_{m,\alpha}} |f_n(x)-(f_n)_{U_{m,\alpha}}|^rg(x)\,dx
\lesssim
(1-s_n)^{\frac{r}{q}}\ell_m^{s_nr}
\int_{U_{m,\alpha}}
\left[\mathcal{D}^{s_n}_{q,U_{m,\alpha}}(f_n)(x)\right]^{r}
\mathcal{M}(g\mathbf{1}_{U_{m,\alpha}})(x)\,dx.
\end{align*}
This, together with \eqref{shenglv} and Lemma \ref{lem:decom}(iv), implies that, for any $n\in\mathbb{N}\cap[n_m,\infty)$,
\begin{align*}
\int_\Omega |f_n(x)-E_mf_n|^r g(x)\,dx
\lesssim
(1-s_n)^{\frac{r}{q}}\ell_m^{s_nr}
\int_{\Omega}
\left[\mathcal{D}^{s_n}_{q,\Omega}(f_n)(x)\right]^{r}
\mathcal{M} g(x)\,dx.
\end{align*}
Using this, \eqref{fanpan},
the assumption that $\mathcal{M}$ is bounded on $Y'$, and  Lemma \ref{xholder}, we obtain, for any $n\in\mathbb{N}\cap[n_m,\infty)$,
\begin{align*}
\|f_n-E_mf_n\|_{X(\Omega)}^r
&\lesssim
\sup_{\genfrac{}{}{0pt}{}{g\ge0}{\|g\|_{Y'}=1}}(1-s_n)^{\frac{r}{q}}\ell_m^{s_nr}
\int_{\Omega}
\left[\mathcal{D}^{s_n}_{q,\Omega}(f_n)(x)\right]^{r}
\mathcal{M} g(x)\,dx\\
&\leq
\sup_{\genfrac{}{}{0pt}{}{g\ge0}{\|g\|_{Y'}=1}}(1-s_n)^{\frac{r}{q}}\ell_m^{s_n r}\left\|\mathcal{D}^{s_n}_{q,\Omega}(f_n)\right\|_{X(\Omega)}^r\|\mathcal{M} g\|_{Y'}\\
&\lesssim
(1-s_n)^{\frac{r}{q}}\ell_m^{\frac{r}{2}}\left\|\mathcal{D}^{s_n}_{q,\Omega}(f_n)\right\|_{X(\Omega)}^r,
\end{align*}
where we used the range $s_n\in (\frac{1}{2},1)$ and $\ell_m\leq \ell_0<1$ in the last step. This proves \eqref{biaoqing}.

Let $\epsilon\in(0,\infty)$.
By \eqref{biaoqing} and \eqref{youjie2}, we conclude that there exist sufficiently large integers  $m_0$ and $n_{m_0}$
such that, for any $n\in\mathbb{N}\cap [n_{m_0},\infty),$
\begin{align}\label{jineng}
\|f_n-E_{m_0}f_n\|_{X(\Omega)}
\lesssim
(1-s_n)^{\frac{1}{q}}\ell_{m_0}^{\frac{1}{2}}\left\|\mathcal{D}^{s_n}_{q,\Omega}(f_n)\right\|_{X(\Omega)}<\epsilon.
\end{align}
Since $\{E_{m_0}f_n\}_{n\in\mathbb{N}}$ is totally bounded in $X(\Omega)$, we deduce that there exists $\{u_{j}\}_{j=1}^L\subset X(\Omega)$ such that
\begin{align*}
\{E_{m_0}f_n\}_{n\in\mathbb{N}}
\subset
\bigcup_{j=1}^LB(u_j,\epsilon).
\end{align*}
By this and \eqref{jineng}, we find that
\begin{align*}
\{f_n\}_{n\in\mathbb{N}\cap [n_{m_0},\infty)}
\subset
\bigcup_{j=1}^LB(u_j,2\epsilon).
\end{align*}
This implies that $\{f_n\}_{n\in\mathbb{N}}$ is totally bounded in $X(\Omega).$
Since $X(\Omega)$ is  complete,
it follows that $\{f_n\}_{n\in\mathbb{N}}$ is relatively compact in $X(\Omega)$.

Let $\{f_{n_k}\}_{k\in\mathbb N}$ be a subsequence of $\{f_n\}_{n\in\mathbb{N}}$ such that
\begin{align}\label{chongs}
\lim_{k\to\infty}f_{n_k}=f
\qquad\text{in }X(\Omega).
\end{align}
Let $j\in\{1,\ldots,N\}$,
$\varphi\in C_c^\infty(\Omega)$,
and  $e_j$ be the $j$-th standard basis vector of $\mathbb{R}^N$.
By  \eqref{chongs} and Lemma \ref{xholder}, we find that
\begin{align}\label{linrui}
\lim_{k\to\infty}\int_\Omega f_{n_k}(x)\,\partial_j\varphi(x)\,dx=
\int_\Omega f(x)\,\partial_j\varphi(x)\,dx.
\end{align}
Let $d$ be sufficiently small such that  $\mathrm{supp}\,(\varphi)+B(\mathbf{0},d)\subset \Omega.$
For any $k\in\mathbb{N}$ and $x\in\mathbb{R}^{N},$ let
\begin{align*}
T_{n_k,e_j,d}\varphi(x)
:=
\frac{2\sigma(S^{N-1})}{K(1,N)}
\int_{\{h\in\mathbb R^N:\ h\cdot e_j\ge0,\ |h|<d\}}
\frac{\varphi(x+h)-\varphi(x)}{|h|}\,\rho_{n_k}(|h|)\,dh.
\end{align*}
By
Lemma  \ref{lem:dir}, we conclude that
$T_{n_k,e_j}\varphi$ uniformly converges to $\partial_j \varphi$
as $k\to\infty.$
Observe that
\begin{align*}
\sup_{x\in\mathbb{R}^{N}}\left|T_{n_k,e_j}\varphi(x)-T_{n_k,e_j,d}\varphi(x)\right|
\lesssim
\int_{|h|\ge d}\rho_{n_k}(|h|)\,dh\to0
\end{align*}
as $k\to\infty.$
Thus, $T_{n_k,e_j,d}\varphi$ uniformly converges to $\partial_j \varphi$
as $k\to\infty.$
This implies that
\begin{equation*}
\lim_{k\to\infty}\left\|T_{n_k,e_j,d}\varphi-\partial_j\varphi\right\|_{X'}=0,
\end{equation*}
which, combined with Lemma \ref{xholder}, further implies that
\begin{align*}
\lim_{k\to\infty}\int_\Omega f_{n_k}(x)\,\partial_j\varphi(x)\,dx=
\lim_{k\to\infty}\int_\Omega f_{n_k}(x)T_{n_k,e_j,d}\varphi(x)\,dx.
\end{align*}
From this and \eqref{linrui}, we infer that
\begin{equation}\label{guiyihua2}
\int_\Omega f(x)\,\partial_j\varphi(x)\,dx
=
\lim_{k\to\infty}
\int_\Omega f_{n_k}(x)
T_{n_k,e_j,d}\varphi(x)\,dx.
\end{equation}
By the fact that  $\mathrm{supp}\,(\varphi)+B(\mathbf{0},d)\subset \Omega,$ we have, for any $k\in\mathbb{N},$
\begin{align}\label{guiyihua}
&\left|\int_\Omega f_{n_k}(x)T_{n_k,e_j,d}\varphi(x)\,dx\right|\notag\\
&\quad=
\frac{2\sigma(S^{N-1})}{K(1,N)}
\left|\int_{\Omega}\varphi(x)
\int_{\{h\in\mathbb{R}^{N}: h\cdot e_j\geq 0,|h|<d\}}
\frac{ f_{n_k}(x-h)- f_{n_k}(x)}{|h|}\rho_{n_k}(|h|)\,dh\,dx\right|\notag\\
&\quad\leq
\frac{2\sigma(S^{N-1})}{K(1,N)}\int_{\Omega}|\varphi(x)|\int_\Omega
\frac{|f_{n_k}(y)-f_{n_k}(x)|}{|x-y|}\rho_{n_k}(|x-y|)\,dy\,dx.
\end{align}
Using H\"older's inequality, we obtain, for any $k\in\mathbb{N}$ and $x\in \mathrm{supp}\,(\varphi),$
\begin{align*}
&\int_\Omega
\frac{|f_{n_k}(y)-f_{n_k}(x)|}{|x-y|}\rho_{n_k}(|x-y|)\,dy\\
&\quad\leq
\left[\int_{\Omega}\rho_{n_k}(|x-y|)\,dy\right]^{\frac{1}{q'}}
\left[\int_\Omega
\frac{|f_{n_k}(y)-f_{n_k}(x)|^q}{|x-y|^q}\rho_{n_k}(|x-y|)\,dy\right]^{\frac{1}{q}}\\
&\quad\leq
\left[\frac{q(1-s_{n_k})}{\sigma(S^{N-1})R^{q(1-s_{n_k})}}\int_{\Omega}
\frac{|f_{n_k}(x)-f_{n_k}(y)|^q}{|x-y|^{N+s_{n_k}q}}\,dy\right]^{\frac{1}{q}}
\end{align*}
From this, \eqref{guiyihua}, and Lemma \ref{xholder}, we deduce that,  for any $k\in\mathbb{N},$
\begin{align*}
\left|
\int_\Omega f_{n_k}(x)T_{n_k,e_j,d}\varphi(x)\,dx
\right|
&\leq
\frac{2[\sigma(S^{N-1})]^{1-\frac{1}{q}}q^{\frac{1}{q}}}{K(1,N)R^{1-s_{n_k}}}(1-s_{n_k})^{\frac{1}{q}}\int_{\Omega}|\varphi(x)|\mathcal{D}^{s_{n_k}}_{q,\Omega}(f_{n_k})(x)\,dx\\
&\leq
\frac{2[\sigma(S^{N-1})]^{1-\frac{1}{q}}q^{\frac{1}{q}}}{K(1,N)R^{1-s_{n_k}}}(1-s_{n_k})^{\frac{1}{q}}
\left\|\mathcal{D}^{s_{n_k}}_{q,\Omega}(f_{n_k})\right\|_{X(\Omega)}\|\varphi\|_{X'}.
\end{align*}
This, combined with \eqref{guiyihua2}, implies that, for any $\varphi\in C^\infty_{\mathrm{c}}(\Omega),$
\begin{equation}\label{guiyihua3}
\left|
\int_\Omega f(x)\,\partial_j\varphi(x)\,dx
\right|
\leq
\frac{2[\sigma(S^{N-1})]^{1-\frac{1}{q}}q^{\frac{1}{q}}}{K(1,N)}
\varliminf_{k\to\infty}(1-s_{n_k})^{\frac{1}{q}}\left\|\mathcal{D}^{s_{n_k}}_{q,\Omega}(f_{n_k})\right\|_{X(\Omega)}\|\varphi\|_{X'}.
\end{equation}
For any $\varphi\in C^\infty_{\mathrm{c}}(\Omega),$ define
\begin{align*}
L_j(\varphi):=-\int_\Omega f(x)\,\partial_j\varphi(x)\,dx.
\end{align*}
Then \eqref{guiyihua3} shows that $L_j$ is a bounded linear  functional on $C^\infty_{\mathrm{c}}(\Omega)\subset X'$. By the Hahn--Banach theorem, $L_j$ extends to a bounded linear functional on  $X'$,
which is denoted by $\widetilde{L_j}$.
Using Lemmas \ref{xingxing} and \ref{xduiou} and the assumption that $X'$ has an absolutely continuous norm,
we conclude that $(X')^\ast=X''=X$. Thus,
there exists $g_j\in (X')^\ast=X''=X$ such that, for any $h\in X'$,
\begin{align*}
\widetilde{L_j}(h)=\int_{\mathbb R^N} g_j(x)\,h(x)\,dx
\end{align*}
and
\begin{align*}
\|g_j\|_{X}
\leq
\frac{2[\sigma(S^{N-1})]^{1-\frac{1}{q}}q^{\frac{1}{q}}}{K(1,N)}
\varliminf_{k\to\infty}(1-s_{n_k})^{\frac{1}{q}}\left\|\mathcal{D}^{s_{n_k}}_{q,\Omega}(f_{n_k})\right\|_{X(\Omega)}.
\end{align*}
From the definition of weak derivatives, we  infer that $\partial^j f=g_j|_{\Omega}.$ This finishes the proof of Theorem \ref{thm-domain}.
\end{proof}
\subsection{Proofs of Theorems \ref{thm:frac-poincare} and \ref{thm:weighted-poincare}}\label{pro:1.2}
In this subsection, we prove Theorems \ref{thm:frac-poincare} and \ref{thm:weighted-poincare}.
Using \cite[Lemma 3.11]{dgpyyz24}
and repeating the proof of \cite[Proposition 2.8]{dlyyz23},
we obtain the following conclusion; we omit the details here.
\begin{proposition}\label{x-poincare}
Let $r\in[1,\infty)$,
$N\in\mathbb{N}\cap[2,\infty)$, and
$\Omega$ be a bounded Lipschitz domain.
Let  $X$
and $X^{\frac{1}{r}}$  be a BBF  space. Assume  that the Hardy--Littlewood maximal operator $\mathcal{M}$ is bounded on $(X^{\frac{1}{r}})'$. Then there exists a   positive constant $C_{X,\Omega}$, depending only on $X$ and $\Omega$, such that, for any $f\in W^{1,X}(\Omega)$,
\begin{align*}
\|f-f_\Omega\|_{X(\Omega)}
\leq
C_{X,\Omega}\|\,|\nabla f|\,\|_{X(\Omega)}.
\end{align*}
\end{proposition}
Now, we prove Theorem \ref{thm:frac-poincare}.
\begin{proof}[Proof of Theorem \ref{thm:frac-poincare}]
We argue by contradiction.
Assume that \eqref{eq:fractional} is false. Then there exist
a monotonically increasing sequence $\{s_k\}_{k\in\mathbb{N}}\subset (0,1)$
and a sequence $\{f_k\}_{k\in\mathbb{N}}\subset X(\Omega)$ such that $\lim_{k\to\infty}s_k=1$ and, for any $k\in\mathbb{N},$
\begin{align}\label{weishenme}
\|f_k-(f_k)_\Omega\|_{X(\Omega)}
>
D\left\|
\left[(1-s_k)
\int_\Omega
\frac{|f_k(\cdot)-f_k(y)|^q}{|\cdot-y|^{N+s_kq}}\,dy
\right]^{\frac1q}
\right\|_{X(\Omega)}.
\end{align}
For any $k\in\mathbb{N},$ let
\begin{align*}
u_k:=\frac{f_k-(f_k)_\Omega}{\|f_k-(f_k)_\Omega\|_{X(\Omega)}}.
\end{align*}
By this and \eqref{weishenme}, we have, for any $k\in\mathbb{N}$,
$\|u_k\|_{X(\Omega)}=1$,
$(u_k)_{\Omega}=0,$ and
\begin{align*}
1=\|u_k\|_{X(\Omega)}
>D(1-s_k)^{\frac{1}{q}}\left\|\mathcal{D}^{s_k}_{q,\Omega}u_k\right\|_{X(\Omega)}.
\end{align*}
Using this and applying Theorem \ref{thm-domain}, after passing to a subsequence,
we conclude that
there exists $u\in W^{1,X}(\Omega)$ such that
\begin{align*}
\lim_{k\to\infty}u_k=u\ \ \text{in}\ \ X(\Omega)
\end{align*}
and
\begin{align*}
\|\,|\nabla u|\,\|_{X(\Omega)}
\leq
B_{q,N}\varliminf_{k\to\infty}(1-s_k)^{\frac{1}{q}}\left\|\mathcal{D}^{s_k}_{q,\Omega}u_k\right\|_{X(\Omega)}
\leq \frac{B_{q,N}}{D}.
\end{align*}
Moreover, it is easy to verify that $u_{\Omega}=0$ and $\|u\|_{X(\Omega)}=1.$
From this and Proposition \ref{x-poincare}, it follows that
\begin{align*}
1=\|u\|_{X(\Omega)}=\|u-u_\Omega\|_{X(\Omega)}
\le
C_{X,\Omega}\|\,|\nabla u|\,\|_{X(\Omega)}
\leq
\frac{C_{X,\Omega}B_{q,N}}{D}.
\end{align*}
This contradicts the assumption
$D>C_{X,\Omega}B_{q,N},$
which completes the proof of Theorem \ref{thm:frac-poincare}.
\end{proof}
We next recall the concept of
Muckenhoupt weights
(see, for instance, \cite[Definition 7.1.2]{g14c}).
\begin{definition}
Let $p\in[1,\infty)$ and $\omega$ be a nonnegative locally integrable
function on $\mathbb{R}^N$. Then
$\omega$ is called an \emph{$A_{p}(\mathbb{R}^N)$ weight}, denoted by
$\omega\in A_{p}(\mathbb{R}^N)$, if, when $p\in(1,\infty)$,
\begin{align*}
[\omega]_{A_{p}(\mathbb{R}^N)}
:=\sup _{\text{cube }Q}\left[
\frac{1}{|Q|}\int_{Q}\omega(x)\,dx\right]\left\{
\frac{1}{|Q|}\int_{Q}[\omega(x)]^{-\frac{1}{p-1}} \,dx\right\}^{p-1}<\infty
\end{align*}
and, when $p=1$,
\begin{align*}
[\omega]_{A_{1}(\mathbb{R}^N)}
:=\sup _{\text{cube }Q}\left[
\frac{1}{|Q|}\int_{Q}\omega(x)\,dx\right]\left\{\mathop{\mathrm{ess\,sup}}_{x\in Q}
[\omega(x)]^{-1}\right\}<\infty.
\end{align*}
In addition, the \emph{class $A_\infty(\mathbb{R}^N)$} is defined by setting
$
A_\infty(\mathbb{R}^N):=\bigcup_{p\in[1,\infty)}A_{p}(\mathbb{R}^N).
$
\end{definition}

Then we present the concept of weighted Lebesgue spaces.

\begin{definition}
Let $p\in[1,\infty)$ and $\omega\in A_{\infty}(\mathbb{R}^N)$.
Let $\Omega$ be an open set
of $\mathbb{R}^N.$
The \emph{weighted Lebesgue space} $L_{\omega}^{p}(\Omega)$ is
defined to be the set of all measurable
functions $f$ on $\Omega$ such that
\begin{align*}
\|f\|_{L^p_\omega(\Omega)}
:=\left[\int_{\Omega}|f(x)|^p\omega(x)\,dx\right]
^{\frac{1}{p}}<\infty.
\end{align*}
\end{definition}
Now, we prove Theorem \ref{thm:weighted-poincare}.
\begin{proof}[Proof of Theorem \ref{thm:weighted-poincare}]
By \cite[Theorem 1.34]{r87}, we find that $L^p_\omega(\mathbb{R}^N)$ has an absolutely continuous norm.
By the assumption that  $\omega\in A_{p}(\mathbb{R}^N)$ and \cite[Proposition 7.1.5(4)]{g14c}, we conclude   that $\omega^{1-p'}\in A_{p'}(\mathbb{R}^N).$
From this and the boundedness of the Hardy--Littlewood maximal operator on weighted Lebesgue spaces (see, for instance, \cite[Theorem 7.1.9]{g14c}), we deduce that $\mathcal{M}$ is   bounded on $L^{p'}_{\omega^{1-p'}}(\mathbb{R}^N).$
Moreover, by the standard duality theory of the Lebesgue space, we find that
$[L^p_{\omega}(\mathbb{R}^N)]^{'}
=L^{p'}_{\omega^{1-p'}}(\mathbb{R}^N).$
This implies that $\mathcal{M}$ is  bounded on $[L^p_{\omega}(\mathbb{R}^N)]^{'}.$
Thus, all the assumptions of Theorem \ref{thm:frac-poincare}  hold with $X:=L^p_{\omega}(\mathbb{R}^N)$ and $r:=1.$  Applying Theorem \ref{thm:frac-poincare}   with $X:=L^p_{\omega}(\mathbb{R}^N)$, we conclude the desired conclusion. This then finishes the proof of Theorem \ref{thm:weighted-poincare}.
\end{proof}
\subsection{Proof of Theorem \ref{thm:numan}}\label{sec2.6}
In this subsection, we prove Theorem \ref{thm:numan}.
\begin{definition}
Let $p,q\in (0,\infty)$, $\Omega\subset\mathbb{R}^{N}$ be a domain, and $\omega\in A_\infty(\mathbb{R}^N).$
The \emph{mixed Lebesgue space}
$L^{p,q}_{\omega}(\Omega\times\Omega)$ is defined to be the
set of all  measurable functions $F$ on $\Omega\times\Omega$
such that
$$\|F\|_{L^{p,q}_{\omega}(\Omega\times\Omega)}:=\left\{\int_\Omega
\left[\int_\Omega |F(x,y)|^q\,d y\right]^{\frac{p}{q}}
\omega(x)\,dx\right\}^{\frac{1}{p}}<\infty.
$$
\end{definition}
\begin{definition}\label{kongjian}
Let $p,q,s\in (0,\infty)$, $\Omega\subset\mathbb{R}^{N}$ be a domain, and $\omega\in A_\infty(\mathbb{R}^N).$
The \emph{weighted Triebel--Lizorkin space}
$F^{s,\omega}_{p,q}(\Omega)$ is defined to be the
set of all  measurable functions $u$ on $\Omega$
such that
\begin{align*}
\|u\|_{F^{s,\omega}_{p,q}(\Omega)}
&:=
\|u\|_{L^p_{\omega}(\Omega)}+[u]_{F^{s,\omega}_{p,q}(\Omega)}\\
&:=
\|u\|_{L^p_{\omega}(\Omega)}+\left\{\int_\Omega
\left[\int_\Omega\frac{|u(x)-u(y)|^q}{|x-y|^{N+sq}}\,d y\right]^{\frac{p}{q}}
\omega(x)\,dx\right\}^{\frac{1}{p}}<\infty.
\end{align*}
Moreover, let
\begin{align*}
F^{s,\omega}_{p,q,0}(\Omega)
:=\left\{u\in F^{s,\omega}_{p,q}(\Omega):\ \int_{\Omega}u(x)\,dx=0\right\}.
\end{align*}
\end{definition}

For any $u\in F^{s,\omega}_{p,q}(\Omega)$,
let
\begin{align*}
J_s(u):=\Phi_s(u)-\int_\Omega f(x)u(x)\,dx.
\end{align*}
\begin{lemma}\label{lem:bianfen}
Let the notation be  as in Theorem \ref{thm:numan}. Then the following  statements hold.
\begin{itemize}
\item [$\mathrm{(i)}$] $F^{s,\omega}_{p,q}(\Omega)$
is a reflexive Banach space and
$F^{s,\omega}_{p,q,0}(\Omega)$
is a closed reflexive  subspace
of $F^{s,\omega}_{p,q}(\Omega)$.
\item [$\mathrm{(ii)}$] $J_s$ is coercive  on
$F^{s,\omega}_{p,q,0}(\Omega);$ that is,
\begin{align*}
\lim_{\|u\|_{F^{s,\omega}_{p,q}(\Omega)}\to\infty}J_s(u)=\infty.
\end{align*}
\item  [$\mathrm{(iii)}$] $J_s$ is weakly lower semicontinuous  on
$F^{s,\omega}_{p,q,0}(\Omega);$ that is, for any sequence $\{u_k\}_{k\in\mathbb{N}}\subset F^{s,\omega}_{p,q,0}(\Omega)$ weakly converging to $u\in F^{s,\omega}_{p,q,0}(\Omega)$,
\begin{align*}
J_s(u)\leq\varliminf_{k\to\infty}J_s(u_k).
\end{align*}
\item  [$\mathrm{(iv)}$]
$J_s$ is  strictly convex on
$F^{s,\omega}_{p,q,0}(\Omega);$ that is, for any $\lambda\in (0,1)$ and $u,v\in F^{s,\omega}_{p,q,0}(\Omega)$ with $u\not=v$,
\begin{align*}
J_s(\lambda u +(1-\lambda)v)
< \lambda J_s(u)+(1-\lambda)J_s(v).
\end{align*}
\item  [$\mathrm{(v)}$] $\Phi_s$ is G\^ateaux differentiable on
$F^{s,\omega}_{p,q,0}(\Omega)$; more precisely,
for any given $u,v\in F^{s,\omega}_{p,q,0}(\Omega)$,
\begin{align*}
\frac{d}{dt}\Phi_s(u+tv)\bigg|_{t=0}=A_s(u,v).
\end{align*}
\end{itemize}
\end{lemma}
\begin{proof}
We first prove (i).
By \cite[Theorem 1(a)]{bp61}, we  conclude  $[L^{p,q}_\omega(\Omega\times\Omega)]^\ast=L^{p',q'}_{\omega^{1-p'}}(\Omega\times\Omega)$, where
$[L^{p,q}_\omega(\Omega\times\Omega)]^\ast$
denotes the dual  space of $L^{p,q}_\omega(\Omega\times\Omega).$
Thus,  $L^{p,q}_\omega(\Omega\times\Omega)$ is  reflexive. Let
\begin{align*}
E:=L^p_{\omega}(\Omega)\times L^{p,q}_\omega(\Omega\times\Omega)
\end{align*}
be equipped with the
\emph{product norm} that, for any $(g,G)\in E,$
\begin{align*}
\|(g,G)\|_E
:=
\|g\|_{L^p_\omega(\Omega)}
+
\|G\|_{L^{p,q}_\omega(\Omega\times\Omega)}.
\end{align*}
Since $L^p_\omega(\Omega)$ and $L^{p,q}_\omega(\Omega\times\Omega)$ are reflexive, it follows that $E$ is reflexive.
For any  $u\in F^{s,\omega}_{p,q}(\Omega)$ and $x,y\in\Omega,$ let
\begin{align*}
T(u)(x,y):=\frac{u(x)-u(y)}{|x-y|^{\frac{N}{q}+s}}\ \ \text{and}\ \ \mathcal{I}(u):=(u,Tu).
\end{align*}
Then $F^{s,\omega}_{p,q}(\Omega)$ is isometrically isomorphic to
$\mathcal{I}(F^{s,\omega}_{p,q}(\Omega))$. On the other hand, it is easy to verify that
$\mathcal{I}(F^{s,\omega}_{p,q}(\Omega))$ is a closed subspace
of $E$. From this and the fact that $E$ is reflexive, we infer that $\mathcal{I}(F^{s,\omega}_{p,q}(\Omega))$ is reflexive.
Thus, $F^{s,\omega}_{p,q}(\Omega)$ is reflexive. This shows (i).

By Theorem \ref{thm:weighted-poincare} and H\"older's inequality,  we find that  there exists a positive constant $C$ such that, for any $u\in F^{s,\omega}_{p,q,0}(\Omega),$
\begin{align*}
J_s(u)\geq
\frac{1}{p}[u]_{F^{s,\omega}_{p,q}(\Omega)}^p-C(1-s)^{\frac{1}{q}}\|f\|_{L^{p'}_{\omega^{1-p'}}(\Omega)}[u]_{F^{s,\omega}_{p,q}(\Omega)}.
\end{align*}
This implies that $J_s$ is coercive on $F^{s,\omega}_{p,q,0}(\Omega)$.
This proves (ii).

Assume that  the sequence $\{u_k\}_{k\in\mathbb{N}}$ weakly
converges to $u\in F^{s,\omega}_{p,q,0}(\Omega).$
Since $T$ is a  bounded linear  map from $F^{s,\omega}_{p,q}(\Omega)$ to $L^{p,q}_\omega(\Omega\times\Omega),$ it follows that the sequence $\{T(u_k)\}_{k\in\mathbb{N}}$ weakly
converges to $T(u)\in L^{p,q}_{\omega}(\Omega\times\Omega).$
By this and the fact that
$\|\cdot\|_{L^{p,q}_{\omega}(\Omega\times\Omega)}$ is weakly lower semicontinuous on $L^{p,q}_{\omega}(\Omega\times\Omega)$, we have
\begin{align*}
\left\{\int_\Omega\left[\int_\Omega\frac{|u(x)-u(y)|^q}{|x-y|^{N+sq}}\,dy\right]^{\frac{p}{q}}\omega(x)\,dx\right\}^{\frac{1}{p}}
=\|T(u)\|_{L^{p,q}_\omega(\Omega\times\Omega)}
\leq
\varliminf_{k\to\infty}\|T(u_k)\|_{L^{p,q}_\omega(\Omega\times\Omega)}.
\end{align*}
This implies that  $J_s(u)$ is  weakly lower semicontinuous on
$F^{s,\omega}_{p,q,0}(\Omega)$.
This shows (iii).

Using  the fact that $\frac{1}{p}t^p$ is strictly
convex  function on $(0,\infty)$ and
the norm $\|\cdot\|_{L^{p,q}_\omega(\Omega\times\Omega)}$ is convex on $L^{p,q}_\omega(\Omega\times\Omega)$, we conclude  that $J_s$ is strictly convex on $F^{s,\omega}_{p,q,0}(\Omega)$.
This proves (iv).

Finally, we show (v).
Let $u,v\in F^{s,\omega}_{p,q,0}(\Omega)$.
If $u\equiv 0$, then
\begin{align*}
\Phi_s(u+tv)=\Phi_s(tv)
=|t|^p\Phi_s(v).
\end{align*}
This implies that
\begin{align*}
\lim_{t\to 0}\frac{\Phi_s(tv)-\Phi_s(0)}{t}
=
\lim_{t\to 0}\frac{|t|^p}{t}\Phi_s(v)=0=A_s(0,v).
\end{align*}
We now assume that $u\not\equiv0$ in $F^{s,\omega}_{p,q,0}(\Omega)$.
Then, by \eqref{1.7x} and \eqref{1.7y}, we have
\begin{align}\label{sanjiaop2}
\frac{\Phi_s(u+tv)-\Phi_s(u)}{t}
=\frac{1}{tp}\int_{\Omega}
\left\{\left[\mathcal{D}^s_{q,\Omega}(u+tv)(x)\right]^p-\left[\mathcal{D}^s_{q,\Omega}(u)(x)\right]^p\right\}\omega(x)\,dx.
\end{align}
Using Minkowski's inequality
of the norm of
$L^q(\Omega)$, we obtain, for any $t\in\mathbb{R}$ and $x\in\Omega,$
\begin{align}\label{sanjiao}
\left|\mathcal{D}^s_{q,\Omega}(u+tv)(x)-\mathcal{D}^s_{q,\Omega}(u)(x)\right|
\leq
|t|\mathcal{D}^s_{q,\Omega}(v)(x).
\end{align}
On the other hand, we conclude that there exists a positive constant $C$ such that, for any $a,b\in(0,\infty)$,
\begin{align*}
|a^p-b^p|
\leq
C(a+b)^{p-1}|a-b|
\end{align*}
and
\begin{align*}
(a+b)^{p-1}b
\leq
C(a^p+b^p).
\end{align*}
From this and \eqref{sanjiao}, we deduce that, for any $t\in(-1,1)$ and $x\in\Omega,$
\begin{align}\label{qishiwang}
&\left|\frac{\left[\mathcal{D}^s_{q,\Omega}(u+tv)(x)\right]^p-\left[\mathcal{D}^s_{q,\Omega}(u)(x)\right]^p}{t}\right|\notag\\
&\quad\leq
C\left[\mathcal{D}^s_{q,\Omega}(u+tv)(x)+\mathcal{D}^s_{q,\Omega}(u)(x)\right]^{p-1}\left[\mathcal{D}^s_{q,\Omega}(v)(x)\right]\notag\\
&\quad\leq
C\left[2\mathcal{D}^s_{q,\Omega}(u)(x)+|t|\mathcal{D}^s_{q,\Omega}(v)(x)\right]^{p-1}\left[\mathcal{D}^s_{q,\Omega}(v)(x)\right]\notag\\
&\quad\lesssim
\left[\mathcal{D}^s_{q,\Omega}(u)(x)\right]^p+\left[\mathcal{D}^s_{q,\Omega}(v)(x)\right]^p.
\end{align}
By this, \eqref{sanjiaop2}, and
the dominated convergence theorem, we obtain
\begin{align}\label{bimian}
\frac{d}{dt}\Phi_s(u+tv)\bigg|_{t=0}
=\frac{1}{p}\int_{\Omega}\frac{d}{dt}[G_x(t)]^{\frac{p}{q}}\bigg|_{t=0}\omega(x)\,dx,
\end{align}
where, for any $t\in\mathbb{R}$ and $x\in\Omega$,
$$
G_x(t)
:=
\int_\Omega
\frac{|u(x)-u(y)+t[v(x)-v(y)]|^q}{|x-y|^{N+sq}}\,dy.
$$
It is easy to verify that, for any $a,b\in\mathbb{R}$,
\begin{align}\label{dengdai}
\left.
\frac{d}{dt}|a+tb|^q
\right|_{t=0}
=
q|a|^{q-2}ab.
\end{align}
From an argument  similar to that used in \eqref{qishiwang} and from H\"older's inequality, we infer that, for any $t\in(-1,1)$ and $x\in\Omega,$
\begin{align*}
&\frac{|G_x(t)-G_x(0)|}{|t|}\\
&\quad\leq
\int_{\Omega}\frac{
\left|
|u(x)-u(y)+t[v(x)-v(y)]|^q
-
|u(x)-u(y)|^q
\right|
}
{|t||x-y|^{N+sq}}\,dy\\
&\quad\lesssim
\int_{\Omega}\frac{|u(x)-u(y)|^{q-1}|v(x)-v(y)|}{|x-y|^{N+sq}}
+
\frac{|v(x)-v(y)|^q}{|x-y|^{N+sq}}\,dy\\
&\quad\leq
\left[
\int_\Omega
\frac{|u(x)-u(y)|^q}{|x-y|^{N+sq}}\,dy
\right]^{\frac{q-1}{q}}
\left[
\int_\Omega
\frac{|v(x)-v(y)|^q}{|x-y|^{N+sq}}\,dy
\right]^{\frac{1}{q}}+\int_{\Omega}\frac{|v(x)-v(y)|^q}{|x-y|^{N+sq}}\,dy.
\end{align*}
Using this, the dominated convergence theorem, and \eqref{dengdai}, we conclude that, for  almost every  $x\in\Omega,$
\begin{align*}
\frac{d}{dt}G_{x}(t)\bigg|_{t=0}
=
q
\int_\Omega
\frac{
|u(x)-u(y)|^{q-2}[u(x)-u(y)][v(x)-v(y)]
}
{|x-y|^{N+sq}}
\,dy.
\end{align*}
This,  together with \eqref{bimian}, implies that
\begin{align*}
\frac{d}{dt}\Phi_s(u+tv)\bigg|_{t=0}
&=\frac{1}{q}\int_\Omega[G_{x}(0)]^{\frac{p}{q}-1}G_{x}^{'}(0)\omega(x)\,dx
=\int_\Omega
\omega(x)
\left[
\int_\Omega
\frac{|u(x)-u(y)|^q}{|x-y|^{N+sq}}\,dy
\right]^{\frac{p}{q}-1}
\\
&\quad\times
\int_\Omega
\frac{
|u(x)-u(y)|^{q-2}[u(x)-u(y)][v(x)-v(y)]
}
{|x-y|^{N+sq}}
\,dy\,dx\\
&=A_s(u,v).
\end{align*}
This proves (v), which  completes the proof of Lemma \ref{lem:bianfen}.
\end{proof}
Now, we prove Theorem \ref{thm:numan}.
\begin{proof}[Proof of Theorem \ref{thm:numan}]
By the  standard variational method and Lemma \ref{lem:bianfen}, we conclude that there exists  a  unique weak solution of \eqref{qiangjie}.
For the completeness of the proof, we give some details here.
Let
\begin{align*}
M:=\inf_{u\in F^{s,\omega}_{p,q,0}(\Omega)}J_s(u).
\end{align*}
Since $J_s$ is coercive on $F^{s,\omega}_{p,q,0}(\Omega)$ (Lemma \ref{lem:bianfen}(ii)),
we deduce that $M\in (-\infty,\infty)$
and there exists a   bounded sequence
$\{u_k\}_{k\in\mathbb{N}}\subset F^{s,\omega}_{p,q,0}(\Omega)$ such that $\lim_{k\to\infty}J_s(u_k)=M.$
By Lemma \ref{lem:bianfen}(i), we find that $F^{s,\omega}_{p,q,0}(\Omega)$
is reflexive. Using this and  the  Eberlein--\v{S}mulian theorem, we conclude that there exist a subsequence $\{u_{k_i}\}_{i\in\mathbb{N}}$
and a function  $u\in F^{s,\omega}_{p,q,0}(\Omega)$
such that $\{u_{k_i}\}_{i\in\mathbb{N}}$
weakly converges to $u\in F^{s,\omega}_{p,q,0}(\Omega).$
From this and the  fact that
$J_s$ is weakly lower semicontinuous on $F^{s,\omega}_{p,q,0}(\Omega)$ (Lemma \ref{lem:bianfen}(iii)), we infer that
$$
J_s(u)
\leq
\varliminf_{i\to\infty}J_s(u_{k_i})
=M.
$$
Thus,  $J_s$ attains its minimum at $u.$
 Let $v\in F^{s,\omega}_{p,q,0}(\Omega).$
For any $t\in\mathbb{R},$ let
$f(t):=J_s(u+tv).$  Then $f$
has a minimum at $t=0.$
By this and Lemma \ref{lem:bianfen}(v), we  find that
\begin{align*}
0&=f'(0)=
\left.
\frac{d}{dt}J_s(u+tv)
\right|_{t=0}
=
A_s(u,v)-\int_\Omega f(x)v(x)\,dx.
\end{align*}
This proves that $u$ is a weak solution of \eqref{eq:weak}.

Next,  we prove uniqueness by using the strict convexity of $J_s$ on $F^{s,\omega}_{p,q,0}(\Omega).$
Let $u$ be a weak solution of \eqref{eq:weak}.
We first show that   $J_s$ attains its minimum at $u.$
Let $v\in F^{s,\omega}_{p,q,0}(\Omega)$. For any $t\in \mathbb{R}$, let $g(t):= J_s(u+t(v-u)).$
Using the fact that $J_s$ is strictly convex on  $F^{s,\omega}_{p,q,0}(\Omega)$ (Lemma \ref{lem:bianfen}(iv)), we conclude that $g$ is  strictly convex on $\mathbb{R}.$ This  implies that, for any $t\in (0,1)$,
\begin{align*}
g(t)\leq (1-t)g(0)+tg(1)
\end{align*}
and hence
\begin{align*}
\frac{g(t)-g(0)}{t}
\leq
g(1)-g(0).
\end{align*}
Letting $t\to 0^+$ and using Lemma \ref{lem:bianfen}(v), we  obtain
\begin{align*}
0=\left.
\frac{d}{dt}J_s(u+t(v-u))
\right|_{t=0}
\leq g(1)-g(0)=J_s(v)-J_s(u).
\end{align*}
This proves that $J_s$ has a minimum at $u.$
Now, assume that $u_1$ and $u_2$
are weak solutions of \eqref{eq:weak}. If $u_1\not=u_2,$ then $J_s(u_1)=J_s(u_2)=M.$
Since $J_s$ is strictly convex on  $F^{s,\omega}_{p,q,0}(\Omega)$ (Lemma \ref{lem:bianfen}(iv)),
it follows that
\begin{align*}
J_s\left(\frac{u_1+u_2}{2}\right)
<
\frac{1}{2}J_s(u_1)+\frac{1}{2}J_s(u_2)
=M.
\end{align*}
This is impossible. Thus, $u_1=u_2.$ This proves the uniqueness of the weak solution
of \eqref{qiangjie}.

Finally, we show \eqref{huxiao}.  Let $u$ be the weak solution of \eqref{eq:weak}.
Letting $v:=u$ in
\eqref{eq:weak}, we obtain
\begin{align*}
\int_{\Omega}f(x)u(x)\,dx
=A_s(u,u)
=[u]_{F^{s,\omega}_{p,q}(\Omega)}^p.
\end{align*}
From this, H\"older's inequality, and Theorem \ref{thm:weighted-poincare}, we deduce that
\begin{align*}
[u]_{F^{s,\omega}_{p,q}(\Omega)}^p
\leq
\|f\|_{L^{p'}_{\omega^{1-p'}}(\Omega)}
\|u\|_{L^p_\omega(\Omega)}
\leq
E (1-s)^{\frac{1}{q}}
\|f\|_{L^{p'}_{\omega^{1-p'}}(\Omega)}[u]_{F^{s,\omega}_{p,q}(\Omega)}.
\end{align*}
Thus,
\begin{align}\label{liliang}
[u]_{F^{s,\omega}_{p,q}(\Omega)}
\leq
E^{\frac{1}{p-1}}
(1-s)^{\frac{1}{q(p-1)}}
\|f\|_{L^{p'}_{\omega^{1-p'}}(\Omega)}^{\frac{1}{p-1}}.
\end{align}
This, combined with Theorem \ref{thm:weighted-poincare}, implies that
\begin{align*}
\|u\|_{L^p_{\omega}(\Omega)}
\leq
E(1-s)^{\frac{1}{q}}[u]_{F^{s,\omega}_{p,q}(\Omega)}
\leq
E^{\frac{p}{p-1}}(1-s)^{\frac{p}{q(p-1)}}
\|f\|_{L^{p'}_{\omega^{1-p'}}(\Omega)}^{\frac{1}{p-1}}.
\end{align*}
By this and \eqref{liliang},
we conclude that \eqref{huxiao} holds. This finishes the proof of Theorem \ref{thm:numan}.
\end{proof}
\section{Applications to Specific Function Spaces\label{shijie}}
In this section, we apply
Theorems \ref{thm-domain} and \ref{thm:frac-poincare},
respectively, to six concrete examples of ball
Banach function spaces, namely
weighted Lebesgue spaces
(Subsection \ref{ws}),
Morrey spaces
(Subsection \ref{mrs}), mixed-norm Lebesgue
spaces (Subsection \ref{ms}),
variable Lebesgue spaces Subsection
\ref{vs}),  Orlicz spaces
(Subsection \ref{os}),  and
Orlicz-slice
spaces (Subsection \ref{oss}).
All the obtained results are new.
These examples clearly demonstrate the universality and practical utility of our results; their extension to other (new) function spaces is evidently possible.
\subsection{Weighted Lebesgue Spaces}\label{ws}
Recall that the definition of weighted Lebesgue spaces is   given in Subsection \ref{pro:1.2}.
It has been pointed out
in \cite[Section 7.1]{shyy} that the weighted Lebesgue space
is a ball Banach function space but may not be a Banach function space. By Theorem \ref{thm-domain} and
the proof of Theorem \ref{thm:weighted-poincare},
we immediately have the following  theorem; we omit the details  here.
\begin{theorem}\label{thm:ws}
Let $p\in (1,\infty)$ and
$\omega\in A_{p}(\mathbb{R}^N)$. Then Theorem \ref{thm-domain} holds  with $X:=L^p_{\omega}(\mathbb{R}^N).$
\end{theorem}
\begin{remark}
To the best of our knowledge, Theorem  \ref{thm:ws} is  new.
\end{remark}
\subsection{Morrey Spaces}\label{mrs}
Morrey spaces were introduced in 1938 by Morrey
\cite{Mo} in order to study the regularity of
solutions to partial differential equations. They
have important applications in the theory of
elliptic partial differential equations,
potential theory, and harmonic analysis (see,
for example, \cite{HMS20,HMS17,HSS18,HS17,s08,s09,s10,st07}).
\begin{definition}
For any $0<\alpha\leq p<\infty$, the \emph
{Morrey space $M_\alpha^p$} is defined to be the set of all the
measurable functions $f$ on $\mathbb{R}^{N}$ with the finite semi-norm
$$\|f\|_{M_\alpha^p}
:=\sup_{B\in\mathbb{B}}|B|^{\frac{1}{p}-\frac{1}{\alpha}}\|f\|_{L^\alpha(B)}.
$$
\end{definition}
\begin{remark}
As was indicated in \cite[p.\,86]{shyy},
the Morrey space $M^p_\alpha$ for any $1\leq \alpha\leq p<\infty$
is a ball 	Banach function space, but is not
a Banach function space in the terminology of
Bennett and Sharpley \cite{bs88}.
\end{remark}
Applying Theorems \ref{thm-domain} and \ref{thm:frac-poincare}, we have the following conclusions.
\begin{theorem}\label{thm:mrs}
Let $1<\alpha\leq p<\infty$. Then Theorems \ref{thm-domain} and \ref{thm:frac-poincare} hold with $X:=M^p_\alpha.$
\end{theorem}
\begin{proof}
By \cite[Theorem 4.1]{st15},
we conclude that $(M^p_\alpha)^{'}=\mathcal{B}^{p'}_{\alpha'}$
and hence $(\mathcal{B}^{p'}_{\alpha'})'
=M^p_\alpha$, where $\mathcal{B}^{p'}_{\alpha'}$ denotes the block space (see \cite[Subsection 1.2]{st15} for the precise definition).
By \cite[Proposition 1.1]{st15}, we find that $(\mathcal{B}^{p'}_{\alpha'})^\ast=M^p_\alpha$. Thus, $(\mathcal{B}^{p'}_{\alpha'})^\ast=(\mathcal{B}^{p'}_{\alpha'})^{'}
=M^p_\alpha$. This, together with Lemma \ref{xingxing}, implies that $\mathcal{B}^{p'}_{\alpha'}$ has
an absolutely  continuous norm.
Moreover, by \cite[Theorem 3.1]{ch14}, we conclude that
$\mathcal{M}$ is  bounded on $\mathcal{B}^{p'}_{\alpha'}$. Thus, all the assumptions of Theorems \ref{thm-domain} and \ref{thm:frac-poincare} hold with $X:=M^p_\alpha$. Applying these, we obtain
the desired conclusions.
This finishes the proof of Theorem \ref{thm:mrs}.
\end{proof}
\begin{remark}
To the best of our knowledge, Theorem  \ref{thm:mrs} is  new.
\end{remark}
\subsection{Mixed-Norm Lebesgue Spaces\label{ms}}
\begin{definition}
For a given vector $\vec{p}:=(p_1,\ldots,p_N)
\in(0,\infty]^N$, the \emph{mixed-norm Lebesgue
space $L^{\vec{p}}$} is defined to be the
set of all measurable functions $f$ on
$\mathbb{R}^N$ with the finite quasi-norm
$$
\|f\|_{L^{\vec{p}}}:=\left\{\int_{\mathbb{R}}
\cdots\left[\int_{\mathbb{R}}|f(x_1,\ldots,
x_N)|^{p_1}\,dx_1\right]^{\frac{p_2}{p_1}}
\cdots\,dx_N\right\}^{\frac{1}{p_N}},
$$
where the usual modifications are made when $p_i=
\infty$ for some $i\in\{1,\ldots,N\}$.
\end{definition}
From the definition of $L^{\vec{p}}$, we easily deduce that
$L^{\vec{p}}$, where $\vec{p}\in(0,\infty)^N$,
is a ball quasi-Banach function space.
But, $L^{\vec{p}}$ may not be a quasi-Banach function space
(see, for instance, \cite[Remark 7.20]{zwyy}).
The study of mixed-norm Lebesgue spaces
can be traced back to
Benedek and Panzone \cite{bp61}.
For more studies
on  mixed-norm Lebesgue spaces,
we refer  to
\cite{cgn17,hlyyjga19,hlyypams,hy}
for the Hardy spaces associated with mixed-norm Lebesgue spaces,
to \cite{cgn17bs, gjn17, GN16,cgn19acha}
for the Triebel--Lizorkin and the Besov spaces
associated with mixed-norm Lebesgue spaces, and
to \cite{cgn17,CGN19MS,cg}
for the (anisotropic) mixed-norm Lebesgue spaces.

By an argument similar to that used in the proof of \cite[Theorem 5.6]{dlyyz23}
and applying Theorems \ref{thm-domain} and \ref{thm:frac-poincare}, we have the following conclusion; we omit the details here.
\begin{theorem}\label{thm:ms}
Let $\vec{p}\in(1,\infty)^N.$
Then Theorems \ref{thm-domain} and \ref{thm:frac-poincare} hold with $X:=L^{\vec{p}}.$
\end{theorem}
\begin{remark}
To the best of our knowledge, Theorem  \ref{thm:ms} is  new.
\end{remark}
\subsection{Variable Lebesgue Spaces\label{vs}}
Let $p:\ \mathbb{R}^N\to[0,\infty)$ be a nonnegative
measurable function. Let
$$
\widetilde{p}_-:=\underset{x\in\mathbb{R}^N}{\mathrm{ess\,inf}}\,p(x)\ \text{and}\
\widetilde p_+:=\underset{x\in\mathbb{R}^N}{\mathrm{ess\,sup}}\,p(x).
$$
A function $p:\ \mathbb{R}^N\to[0,\infty)$ is said to be \emph{globally
log-H\"older continuous} if there exist
 $p_{\infty}\in\mathbb{R}$
and a positive constant $C$ such that, for any
$x,y\in\mathbb{R}^N$,
$$
|p(x)-p(y)|\le \frac{C}{\log(e+1/|x-y|)}\ \ \text{and}\ \
|p(x)-p_\infty|\le \frac{C}{\log(e+|x|)}.
$$
The \emph{variable Lebesgue space
$L^{p(\cdot)}$} associated with the function
$p:\ \mathbb{R}^N\to[0,\infty)$ is defined to be the set
of all  measurable functions $f$ on $\mathbb{R}^N$ with
the finite quasi-norm
$$
\|f\|_{L^{p(\cdot)}}:=\inf\left\{\lambda
\in(0,\infty):\ \int_{\mathbb{R}^N}\left[\frac{|f(x)|}
{\lambda}\right]^{p(x)}\,dx\le1\right\}.
$$
It is known that $L^{p(\cdot)}$ is a ball quasi-Banach function space
(see, for instance, \cite[Section 7.8]{shyy}). By an argument similar to that used in the proof of \cite[Theorem 5.11]{dlyyz23}
and applying Theorems \ref{thm-domain} and \ref{thm:frac-poincare}, we have the following conclusion; we omit the details here.
\begin{theorem}\label{thm:vs}
Let $p:\ \mathbb{R}^N\to[0,\infty)$ be globally
log-H\"older continuous with
$1< \widetilde{p}_-\leq  \widetilde{p}_+<\infty$. Then Theorems \ref{thm-domain} and
\ref{thm:frac-poincare}   hold with $X:=L^{p(\cdot)}.$
\end{theorem}
\begin{remark}
To the best of our knowledge, Theorem  \ref{thm:vs} is  new.
\end{remark}
\subsection{Orlicz Spaces}\label{os}
Recall that a non-decreasing function $\Phi:\ [0,\infty)\to[0,\infty)$
is called an \emph{Orlicz function} if it
satisfies $\Phi(0)= 0$, $\Phi(t)>0$ for any $t\in(0,\infty)$,
and $\lim_{t\to\infty}\Phi(t)=\infty$.
Moreover, $\Phi$ is said to be
of \emph{upper} (resp. \emph{lower})
\emph{type} $p$ with $p\in(-\infty,\infty)$ if
there exists a positive constant $C_{p}$, depending on $p$,
such that, for any $t\in[0,\infty)$
and $s\in[1,\infty)$ [resp. $s\in (0,1)$],
$
\Phi(st)\le C_{p}s^p \Phi(t).
$
Then we present the concept of Orlicz spaces;
see, for instance, \cite[p.\,58, Definition 2]{rr}.

\begin{definition}
Let $\Phi$ be an Orlicz function with lower type
$p_{\Phi}^-\in(0,\infty)$ and upper type $p_{\Phi}^+\in(0,\infty)$.
The \emph{Orlicz space $L^\Phi$} is defined
to be the set of all measurable
functions $f$ on $\mathbb{R}^n$ such that
$$\|f\|_{L^\Phi}
:=\inf\left\{\lambda\in(0,\infty):\ \int_{\mathbb{R}^N}
\Phi\left(\frac{|f(x)|}{\lambda}\right)\,dx\le1\right\}$$
is finite.
\end{definition}
It is easy to prove that the Orlicz space $L^\Phi$
is a quasi-Banach
function space and hence
a ball quasi-Banach
function space (see \cite[Subsection 7.6]{shyy}).
By an argument similar to that used in the proof of \cite[Theorem 5.20]{dlyyz23}
and applying Theorems \ref{thm-domain} and \ref{thm:frac-poincare}, we have the following conclusion; we omit the details here.
\begin{theorem}\label{thm:os}
Let $\Phi$ be an Orlicz function with positive lower
type $p_{\Phi}^-\in(1,\infty)$.
Then Theorems \ref{thm-domain} and \ref{thm:frac-poincare}
hold with $X:=L^\Phi.$
\end{theorem}
\begin{remark}
To the best of our knowledge, Theorem  \ref{thm:os} is  new.
\end{remark}
\subsection{Orlicz-Slice  Spaces\label{oss}}
\begin{definition}
Let $\Phi: [0,\infty)\to [0,\infty)$
be an Orlicz function with positive
lower type $p_{\Phi}^-$ and positive upper
type $p_{\Phi}^+$. For any given $t,p\in(0,\infty)$,
the \emph{Orlicz-slice space}
$(E_\Phi^p)_t$ is defined to be the set of all
measurable functions $f$ on $\mathbb{R}^N$ with the finite
quasi-norm
$$
\|f\|_{(E_\Phi^p)_t} :=\left\{\int_{\mathbb{R}^N}
\left[\frac{\|f\mathbf{1}_{B(x,t)}\|_{L^\Phi}}
{\|\mathbf{1}_{B(x,t)}\|_{L^\Phi}}\right]
^p\,dx\right\}^{\frac{1}{p}}.
$$
\end{definition}
According to
both \cite[Lemma 2.28]{zyyw} and \cite[Remark 7.41(i)]{zwyy},
the Orlicz-slice space $(E_\Phi^p)_t$ is a
ball Banach function space, but in general is not a
Banach function space.
The Orlicz-slice spaces were introduced in
\cite{zyyw} as a generalization of
both the slice space of	Auscher and Mourgoglou
\cite{AM2014,APA} and the Wiener amalgam space
in \cite{knt,h19}.  By an argument similar to that used in the proof of \cite[Theorem 5.25]{dlyyz23}
and applying Theorems \ref{thm-domain} and \ref{thm:frac-poincare}, we have the following conclusion; we omit the details here.
\begin{theorem}\label{thm:oss}
Let $t\in (0,\infty)$, $p\in(1,\infty)$, and $\Phi$ be an Orlicz function with positive lower
type $p_{\Phi}^-\in (1,\infty)$.
Then Theorems \ref{thm-domain} and
 \ref{thm:frac-poincare} hold with $X:=(E_\Phi^p)_t.$
\end{theorem}
\begin{remark}
To the best of our knowledge, Theorem  \ref{thm:oss} is  new.
\end{remark}

\medskip

\noindent\textbf{Funding}\quad This project is partially supported by the National
Natural Science Foundation of China (Grant Nos. 12431006, 12371093, and 12501118),
the Beijing Natural Science Foundation (Grant No. 1262011), and
the Fundamental Research Funds for the Central Universities
(Grant No. 2253200028).

\medskip

\noindent\textbf{Data Availability}\quad  Data sharing not applicable
to this article as no datasets were generated or analyzed during
the current study.

\section*{Declarations}

\noindent\textbf{Conflict of interest}\quad The authors declare
that they have no conflict of interest.

\bigskip

\noindent
Xiaosheng Lin

\smallskip

\noindent
School of Mathematical Sciences, Jimei University,
Xiamen 361005, The People's Republic of China

\smallskip

\noindent {\it E-mail}: \texttt{xslin@jmu.edu.cn}

\bigskip

\noindent Dachun Yang (Corresponding author), Wen Yuan and Yangyang Zhang

\smallskip

\noindent Laboratory of Mathematics and Complex Systems
(Ministry of Education of China),
School of Mathematical Sciences, Institute for Advanced Study,
Beijing Normal University,
Beijing 100875, The People's Republic of China

\smallskip

\noindent{\it E-mails:} \texttt{dcyang@bnu.edu.cn} (D. Yang)

\noindent\phantom{{\it E-mails:}} \texttt{wenyuan@bnu.edu.cn} (W. Yuan)

\noindent\phantom{{\it E-mails:}} \texttt{yangyzhang@bnu.edu.cn} (Y. Zhang)

\end{document}